\documentclass{compositio}

\RequirePackage[l2tabu, orthodox]{nag}
\usepackage[T1]{fontenc}
\usepackage[utf8]{inputenc}
\usepackage{lmodern}
\usepackage[english]{babel}
\usepackage[usenames,dvipsnames]{xcolor}
\usepackage{amsmath}
\usepackage[alphabetic]{amsrefs}
\usepackage{enumerate}
\usepackage{url}
\usepackage{appendix}

\newcommand{\nocontentsline}[3]{}
\newcommand{\tocless}[2]{\bgroup\let\addcontentsline=\nocontentsline#1{#2}\egroup}

\makeatletter
\def\l@subsection{\@tocline{2}{0pt}{2pc}{6pc}{}} 
\makeatother

\newcommand{\R}{\mathbb{R}}
\newcommand{\N}{\mathbb{N}}
\newcommand{\Z}{\mathbb{Z}}

\newcommand{\T}{\mathbb{S}^1}

\newcommand{\Sbb}{\mathbb{S}}

\newcommand{\Homeo}{\mathsf{Homeo}}

\newcommand{\cO}{\mathcal{O}}

\newcommand{\PL}{\mathsf{PL}}
\newcommand{\Tsf}{\mathsf{T}}
\newcommand{\Hsf}{\mathsf{H}}

\newcommand{\Fix}{\mathsf{Fix}}
\newcommand{\sub}{\mathsf{Sub}}
\newcommand{\id}{\mathsf{id}}

\newcommand{\rist}{\mathsf{RiSt}}

\DeclareMathOperator{\Out}{\mathrm{Out}}
\DeclareMathOperator{\Aut}{\mathrm{Aut}}
\DeclareMathOperator{\MCG}{\mathrm{MCG}}

\newtheorem{thm}{Theorem}[section]
\newtheorem{introthm}{Theorem}

\newtheorem*{thm*}{Theorem}

\newtheorem*{claim}{Claim}

\newtheorem{lem}[thm]{Lemma}

\newtheorem{prop}[thm]{Proposition}
\newtheorem{cor}[thm]{Corollary}
\newtheorem{intro-cor}[introthm]{Corollary}

\theoremstyle{definition}
\newtheorem{dfn}[thm]{Definition}

\theoremstyle{remark}
\newtheorem{rem}[thm]{Remark}

\newtheorem{ques}{Question}

\usepackage{indentfirst}
\usepackage{microtype}

\usepackage[colorlinks=true,linkcolor=blue,citecolor=magenta]{hyperref}

\begin{document}

\title{Groups of piecewise linear homeomorphisms of  flows}

\author{Nicol\'as Matte Bon}
\email{mattebon@math.univ-lyon1.fr}
\address{CNRS \& Institut Camille Jordan (ICJ, UMR CNRS 5208),
Universit\'e de Lyon,
43 blvd.\ du 11 novembre 1918,
69622 Villeurbanne,
France}

\author{Michele Triestino}
\email{michele.triestino@u-bourgogne.fr}
\address{Institut de Math\'ematiques de Bourgogne (IMB, UMR CNRS 5584),
Universit\'e Bourgogne Franche-Comt\'e,
9 av.~Alain Savary,
21000 Dijon,
France}

\thanks{The second author was partially supported by the project ANR Gromeov (ANR-19-CE40-0007) and the project ANER Agroupes (AAP 2019 R\'egion Bourgogne--Franche--Comt\'e).}

\begin{abstract}
To every dynamical system $(X,\varphi)$ over a totally disconnected compact space, we associate a left-orderable group $T(\varphi)$. It is defined as a group of homeomorphisms of the suspension of $(X,\varphi)$ which preserve every orbit of the suspension flow and act  by dyadic piecewise linear homeomorphisms in the flow direction. We show that if the system is minimal, the group is simple, and if it is a subshift then the group is finitely generated. The proofs of these two statements are short and elementary, providing straightforward examples of  finitely generated simple left-orderable groups. We show that if the system is minimal, every action of the corresponding group on the circle has a fixed point. These constitute the first examples of finitely generated left-orderable  groups with this fixed point property. We show that for every system $(X,\varphi)$, the group $T(\varphi)$ does not have infinite subgroups with Kazhdan's property $(T)$.
 
  In addition, we show that for every minimal subshift, the corresponding group is never finitely presentable. Finally if $(X,\varphi)$ has a dense orbit then the isomorphism type of the group $T(\varphi)$ is a complete invariant of flow equivalence of the pair $\{\varphi, \varphi^{-1}\}$.  
%
\end{abstract}

\maketitle

\section{Introduction}
A group is \emph{left-orderable} if it admits a total order which is invariant under left multiplication. For instance, the group $\Homeo_+(\R)$ of orientation-preserving homeomorphisms of the real line is left-orderable, as well as all of its subgroups. Conversely, if a non-trivial countable group is left-orderable, then it admits a faithful orientation-preserving action on $\R$ without global fixed points. This provides a dynamical point of view on left-orderability. See Deroin, Navas and Rivas \cite{GOD} for a detailed treatment. 

One recurrent difficulty in this dynamical approach  is the lack of compactness of $\R$. As a possible solution to this issue,  Deroin \cite{Deroin} showed that every finitely generated left-orderable group can be realised as a group of homeomorphisms of a compact space $Y$, without global fixed points, which preserves each orbit of a flow $\Phi\colon \R\times Y\to Y$.  The goal of this paper is to explore how this point of view can be conversely used to define new  examples of left-orderable groups. We introduce and study a family of left-orderable groups associated to suspension flows of homeomorphisms of the Cantor set (more generally, of a Stone space). These groups  turn out to satisfy interesting properties in the setting of group actions on one-manifolds:  simplicity and finite generation (Theorems \ref{t-fg} and \ref{t-simple}), and a fixed point property for actions on the circle (Theorem \ref{t-circle}).

 \smallskip

 Throughout the paper, let $X$ be a non-empty Stone space, that is, a totally disconnected compact Hausdorff space, and let $\varphi$ be a self-homeomorphism of $X$. The dynamical system $(X, \varphi)$ will be called a \emph{Stone system}, or a \emph{Cantor system} in the case $X$ is homeomorphic to a Cantor set.  The \emph{suspension} (or \emph{mapping torus})  of  ($X, \varphi)$ is the space $Y^\varphi:=(X\times \R)/\Z$, where the quotient is taken with respect to the diagonal action of  $\Z$ on $X\times \R$ given by $n\cdot (x, t)= (\varphi^n(x), t-n)$. The natural projection of $(x, t)\in X\times \R$ to $Y^\varphi$ is denoted by $[x, t]$.   The suspension comes with a flow $\Phi \colon \R\times Y^\varphi\to Y^\varphi$, given by $\Phi^t([x, s])=[x, s+t]$.

We denote by $\Hsf(\varphi)$  the group of homeomorphisms of $Y^\varphi$, and by $\Hsf_0(\varphi)$ its subgroup  consisting of elements $g\in  \Hsf(\varphi)$ that are isotopic to the identity. This condition forces every element $g\in \Hsf_0(\varphi)$ to preserve every $\Phi$-orbit, since $\Phi$-orbits are precisely the path-components of $Y^\varphi$.\footnote{Conversely, an element preserving every $\Phi$-orbit needs not be isotopic to the identity in general (see \cite{BCE}), but it is if $(X, \varphi)$ is minimal \cite{APP}.  This subtle problem will be irrelevant in this paper, with the exception of the appendix.} Observe that $\Hsf_0(\varphi)$ is usually not closed in $\Hsf(\varphi)$ (with respect to the compact open topology).
A  concrete  characterisation of $\Hsf_0(\varphi)$  is the following (see  Boyle, Carlsen and Eilers   \cite[Theorem 3.1]{BCE}): a homeomorphism $g\in \Hsf(\varphi)$ belongs to $\Hsf_0(\varphi)$ if and only if there exists a continuous function $\tau\colon Y^\varphi\to \R$ such that $g$ is given by
\[g(y)=\Phi^{\tau(y)}(y).\]
The  groups $\Hsf_0(\varphi)$ can be thought as generalisations of the group $\Homeo_+(\T)$ of orientation-preserving homeomorphisms
of the circle (indeed, it reduces to it when $X$ is reduced to a point). For instance, it is a classical result by Schreier and Ulam \cite{SU} that $\Homeo_+(\T)$ is simple, and after a result of Aliste-Prieto and Petite \cite{APP} the group $\Hsf_0(\varphi)$ is also simple as soon as $(X, \varphi)$ is {minimal} (i.e.\ every orbit is dense). On the other hand, unlike $\Homeo_+(\T)$, the group $\Hsf_0(\varphi)$ is often left-orderable (see Proposition \ref{p-lo}). For instance if the system $(X, \varphi)$ has a {dense orbit}, the restriction of the natural action of $\Hsf_0(\varphi)$ to any dense $\Phi$-orbit yields a faithful action on the real line. 

The group $\Homeo_+(\T)$ has a natural finitely generated subgroup which is also simple, namely  Thompson's group $T$ (whose definition is  recalled in Section~\ref{s-group}). Pursuing this analogy, we define a group $\Tsf(\varphi)\le \Hsf_0(\varphi)$.   It is defined as the subgroup of $\Hsf_0(\varphi)$ formed by homeomorphisms that in the direction of the flow are locally defined by dyadic piecewise linear  homeomorphisms, and whose displacement is locally constant in the totally disconnected  direction  (see Definition \ref{d-group}). The goal of this paper is to study properties of the group $\Tsf(\varphi)$.

\subsection{Finite generation and simplicity} Let us begin with a criterion that ensures that the group $\Tsf(\varphi)$ is  finitely generated. Recall that a  \emph{subshift} over a finite alphabet $A$  is a Stone system $(X, \varphi)$ where $X\subset A^\Z$ is a non-empty closed shift-invariant subset, and  $\varphi$ is the restriction of the shift to $X$. (We allow $X$ to admit isolated points, for instance it can be countable). 

\begin{introthm} \label{t-fg}
If the Stone system $(X, \varphi)$ is conjugate to a subshift, the group $\Tsf(\varphi)$ is finitely generated.\end{introthm}

Note that the converse is also true (see Proposition \ref{p-not-fin-gen}).

It has been until recently an open question, attributed to Rhemtulla, whether there exists a finitely generated left-orderable group which is simple (this question first appeared in [Notices AMS \textbf{28} (June 1982), p.\ 327]). The first examples of such groups have been constructed by Hyde and Lodha \cite{HL}. They define a family of groups $G_\rho$ parametrised by a sequence  $\rho\in  \{a^{\pm 1}, b^{\pm 1}\}^\Z$ on a four-letter alphabet, required to be quasi-periodic and to verify a palindromicity condition. Each $G_\rho$ is defined as the group generated by six explicit piecewise linear  homeomorphisms of the real line with infinitely many pieces, which depend in a natural way on the sequence $\rho$. In relation to the present setting,  if we let $(X, \varphi)$ be the subshift given by the orbit-closure of the sequence $\rho \in \{a^{\pm 1}, b^{\pm 1}\}^\Z$, then the group $G_\rho$ can be shown to be isomorphic to a (strict) subgroup of $\Tsf(\varphi)$ (this is close to the construction of Deroin \cite{Deroin}). 

 Note that several natural classes of left-orderable groups are actually \emph{locally indicable}   (that is, every finitely generated non-trivial subgroup surjects to $\Z$). Historically the first  example which is not locally indicable (due to Thurston \cite{Thurston}, rediscovered by Bergman \cite{Bergman}) is the central extension $\widetilde\Gamma_{2,3,7}$ in $\widetilde{\mathsf{PSL}}(2,\R)$ of the $(2,3,7)$-triangle group  in $\mathsf{PSL}(2,\R)$, which is finitely generated and perfect. Another example with the same property is the central extension $\widetilde T$ of Thompson's group $T$  (defined in Section~\ref{s-LO}), see Navas \cite[Remark 2.2.51]{Nav-book}.

Many examples of simple groups arise as groups of homeomorphisms $G\le \Homeo_+(Z)$ of some (Hausdorff) space $Z$ whose action  is \emph{micro-supported}, meaning that  every non-empty open subset  $U\subset Z$ has a non-trivial \emph{rigid stabiliser} $\rist(U)$, defined as the subgroup of elements with support in $U$. This property helps in proving simplicity using a  method \emph{à la Higman--Epstein} \cite{Epstein, Higman} (see Section \ref{ss-simple}). The starting observation is a ``Double Commutator Lemma'' (see Lemma \ref{l-epstein}), stating that every non-trivial normal subgroup of a micro-supported group $G$ must contain the derived subgroup $\rist_G(U)'$ of the rigid stabiliser  of some non-empty open set. As a consequence, to prove simplicity it is enough to show that $G$ is generated by commutators of elements of sufficiently ``small'' support (cf.\ Proposition \ref{p-simplicity}).  However it is easy to see that a \emph{finitely generated} simple group cannot admit a micro-supported action on any \emph{non-compact}, locally compact space (see Remark \ref{r-micro-supp-R}). This is an important technical difficulty when establishing simplicity of finitely generated simple subgroups of $\Homeo_+(\R)$.  In particular the proof of simplicity of the   groups $G_\rho$ given by Hyde and Lodha  \cite{HL}  goes through a quite involved combinatorial  analysis of the action of the groups $G_\rho$ on the real line $\R$.

The groups $\Tsf(\varphi)$ provide a more conceptual way to obtain such examples, which makes this difficulty disappear.  In fact,  the action of the group $\Tsf(\varphi)$ on the \emph{compact} suspension space $Y^\varphi$ is micro-supported. This yields a very short proof of  the following result.

\begin{introthm}\label{t-simple}
If the Stone system $(X, \varphi)$ is minimal, then $\Tsf(\varphi)$ is simple. 
\end{introthm}

Note that minimality is also necessary: if $(X,\varphi)$ is not minimal then $\Tsf(\varphi)$ admits a non-trivial proper quotient (see Corollary \ref{c-not-simple}).

Combining with Theorem \ref{t-fg}, we have:
\begin{intro-cor}
For any infinite minimal subshift $(X, \varphi)$, the group $\Tsf(\varphi)$ is simple, finitely generated, and left-orderable.
\end{intro-cor} 

\begin{rem}
Theorems \ref{t-fg} and \ref{t-simple}  are analogous to  the results of Matui \cite{Mat-simple} on the derived subgroup of the \emph{topological full group} $[[\varphi]]$ of a Cantor minimal system $(X, \varphi)$, and were largely inspired by these (with the difference that here the group $\Tsf(\varphi)$ always coincides with its derived subgroup). Recall that the topological full group $[[\varphi]]$ is  the group of all homeomorphisms of~$X$ that locally coincide with a power of $\varphi$. Note that the group $[[\varphi]]$ is very far from being torsion-free (it contains infinite locally finite subgroups), thus it cannot be left-orderable.
 Another important difference is that  $\Tsf(\varphi)$ always contains non-abelian free subgroups and is therefore non-amenable (see Corollary~\ref{c-free-subgroups}), whereas Juschenko and Monod \cite{JM} proved that the topological full group~$[[\varphi]]$ is amenable. Note that by a result of Witte Morris \cite{WM} (see also Deroin \cite{Deroin}), no amenable left-orderable group can be simple, as such a group is locally indicable. Finally, Grigorchuk and Medynets proved that the topological full group $[[\varphi]]$ of a Cantor minimal system $(X, \varphi)$ is \emph{locally embeddable into finite groups (LEF)} \cite{GM}, but in contrast the group $\Tsf(\varphi)$ is not, as it contains subgroups isomorphic to Thompson's group $F$ (see Section \ref{s-subgroups}; $F$ is not LEF because it is finitely presented and not residually finite, therefore \cite[Proposition 7.3.8]{CC} applies).
\end{rem}

\subsection{Actions of $\Tsf(\varphi)$ on the circle and on the real line}
Our next result analyzes actions of the group $\Tsf(\varphi)$  on the  circle. Clearly every group of homeomorphisms of the real line can also be seen as acting on the circle $\Sbb^1\simeq\R\cup\{\infty\}$ by fixing the point at infinity. The next result says that conversely,  actions  of the group $\Tsf(\varphi)$  on the circle  always reduce to actions on the real line.
 \begin{introthm} \label{t-circle}
	If the Stone system $(X, \varphi)$ is infinite and minimal, every continuous action of the group $\Tsf(\varphi)$ on the circle has a fixed point.
\end{introthm}

 In combination with Theorem \ref{t-fg}, this provides the first examples of non-trivial, finitely generated subgroups of $\Homeo_+(\Sbb^1)$ with this property. 
 
\begin{rem}
	We remark here, and this will be silently used throughout the text, that if $G$ is a perfect group, every one-dimensional continuous action of $G$ preserves the orientation. 
\end{rem}

Theorem \ref{t-circle} implies a related property for actions of the group $\Tsf(\varphi)$ on the real line. Recall (see Deroin, Navas and Rivas \cite[\S 3.5]{GOD}) that every action of a finitely generated group $G$ on the real line by orientation-preserving homeomorphisms, without global fixed points, is exactly of one of the following three types. 
\begin{enumerate}[(I)]
\item \label{i-typeI} The action preserves a non-zero Radon measure on  $\R$.
\item \label{i-typeII} The action is semi-conjugate to a minimal action on $\R$ which commutes with the action of $\Z$ on $\R$ by translations, and passes to the quotient to a \emph{proximal}\footnote{In this context, a minimal $G$-action on $X=\R/\Z$ or $\R$ is \emph{proximal} if for all compact intervals $I, J\subset X$ there exists $g\in G$ such that $g(I)\subset J$.} action on $\Sbb^1=\R/\Z$.
\item \label{i-typeIII} The action is semi-conjugate to a minimal and proximal action on $\R$.
\end{enumerate}
(The notion of semi-conjugacy is recalled in Section \ref{s-LO}.)
While actions of type \eqref{i-typeIII} are in some sense the ``generic'' ones, all finitely generated, left-orderable groups known so far were also known to admit actions of type  \eqref{i-typeI} or \eqref{i-typeII}. This lead Deroin, Navas and Rivas to ask whether there exist finitely generated, left-orderable groups whose actions on the real line, without global fixed points, are {all} of type \eqref{i-typeIII} \cite[Question 3.5.11]{GOD}. 
A discussion on this question can also be found in Navas' survey \cite[see Question 4]{Nav}.

For a given finitely generated group $G$, it is clear from the above trichotomy that the following statements are equivalent.
\begin{enumerate}[(i)]
	\item Every orientation-preserving action of $G$ on the real line, without global fixed point, is of type \eqref{i-typeIII}.
	\item Every orientation-preserving action of $G$ on the circle, which lifts to the real line, has a fixed point.
\end{enumerate}
Thus Theorem \ref{t-circle} answers the question of Deroin, Navas and Rivas in the affirmative.
\begin{intro-cor} \label{c-monster}
If the Stone system $(X, \varphi)$ is infinite and minimal, every action of the group $\Tsf(\varphi)$ on the real line, without global fixed points, is of type \eqref{i-typeIII}. 
\end{intro-cor}

The proof of Theorem \ref{t-circle} is based on  an analogue of a property known to hold for Thompson's group $T$, namely that every non-trivial action of the group $T$ on the circle is semi-conjugate to the standard one (Theorem~\ref{t-TonS1}). Using arguments from a proof of this fact  by Le Boudec and the first named author \cite{LB-MB-subdyn}, we show that every action of $\Tsf(\varphi)$ on $\Sbb^1$ without global fixed points, must give rise to a continuous surjective map $\Sbb^1\to Y^\varphi$. But here the space $Y^\varphi$ is not path-connected, thus such an action cannot exist. The action of the group $\Tsf(\varphi)$ on  the space $Y^\varphi$ (rather than its action on the real line) therefore  plays a crucial role also here. 

\begin{rem}
	After the first version of this work, James Hyde, Yash Lodha, Andr\'es Navas and Crist\'obal Rivas informed us of the work \cite{HLNR} which  independently establishes the analogue of Corollary~\ref{c-monster} for the groups $G_\rho$ of Hyde and Lodha~\cite{HL}.
\end{rem}

We mention the following related question.

\begin{ques}
	Assume that the Stone system $(X,\varphi)$ is infinite minimal. Is every action of $\Tsf(\varphi)$ on the real line, without global fixed points, semi-conjugate to its action on some infinite $\Phi$-orbit?
\end{ques}

At the moment of writing we are unable to answer this question. Let us point out that a similar statement (previously unnoticed?) holds for the central extension $\widetilde T$ of Thompson's group~$T$ (see Theorem \ref{t:tilde-T}), which embeds into  $\Tsf(\varphi)$ as soon as $\varphi$ has infinite order (see Proposition~\ref{p-nonamenable}).

Let us also mention that an example close in spirit  to Theorem \ref{t-circle}  appears in Calegari \cite{forcing}: let $\widetilde\Gamma_{2,3,7}\subset \widetilde{\mathsf{PSL}}(2, \R)$ be the central extension of the $(2,3,7)$-triangle group in $\mathsf{PSL}(2, \R)$; it is not true that every action of $\widetilde{\Gamma}_{2,3,7}$ on the circle has a global fixed point, but this happens for the restriction to $\widetilde{\Gamma}_{2,3,7}$ of actions of a finitely generated, circularly left-orderable overgroup.

\subsection{Non-existence of subgroups with Kazhdan's property $(T)$}
Recall that a countable group $G$ has Kazhdan's property $(T)$ if every $G$-action on a Hilbert space by affine isometries has a fixed point. 
It is an open question whether an infinite group with Kazhdan's property $(T)$ can act faithfully  on a one-dimensional manifold by homeomorphisms.  Navas \cite{Nav-T} proved that this is not the case if one restricts to actions on the circle  by smooth diffeomorphisms. By results in \cite{LMT, Cor-partial}, this extends  to actions by piecewise smooth diffeomorphisms with finitely many pieces. Moreover, every action of a Kazhdan group on the circle by homeomorphisms which are individually smooth on the complement of a countable closed set must have a finite orbit \cite{LMT}.  Here we prove that also the groups $\Tsf(\varphi)$ do not have Kazhdan's property $(T)$.

\begin{introthm}\label{t-T}
For every Stone system $(X, \varphi)$ the group $\Tsf(\varphi)$  does not contain infinite subgroups with Kazhdan's property $(T)$. The same holds more generally for the larger group $\PL(\varphi)$ (Definition~\ref{d-group}).
\end{introthm}

Although we cannot apply the results of \cite{Nav-T, LMT, Cor-partial} to the group $\Tsf(\varphi)$ directly (unless $\varphi$ is periodic), we show that it can be approximated by groups to which such results do apply, and which therefore do not have Kazhdan's property $(T)$. Then
the proof  relies on the fact (proved independently by Shalom \cite[Theorem 6.7]{Shalom} and Gromov \cite[\S 3.8.B]{RWRG}) that  the set of property~$(T)$ groups is open in the space of marked groups (in the sense of Grigorchuk \cite{Grig}). 

\begin{rem}
Since the groups $G_\rho$ defined by Hyde and Lodha \cite{HL} arise as subgroups of $\Tsf(\varphi)$ for some $(X, \varphi)$, Theorem \ref{t-T} also implies that they do not have Kazhdan's property $(T)$. 
\end{rem}
\subsection{Further results}

By the same limit argument used to prove Theorem \ref{t-T}, we also obtain the following in Section~\ref{s-T}.

\begin{introthm} \label{t-fp}
For every minimal subshift $(X, \varphi)$, the group $\Tsf(\varphi)$ is not finitely presented. 
\end{introthm}

In Section \ref{s-isomorphisms}, we clarify how the isomorphism class of the group $\Tsf(\varphi)$ depends on $(X, \varphi)$.  Two Stone systems $(X_1, \varphi_1)$ and $(X_2, \varphi_2)$ are said to be \emph{flow equivalent} if there exists a homeomorphism $q\colon Y^{\varphi_1}\to Y^{\varphi_2}$ between the suspensions which sends orbits to  orbits in an orientation-preserving fashion. By a result of Parry and Sullivan \cite{PS} (sse Theorem \ref{t-PS}), this is equivalent to the existence of clopen subsets $X'_i\subset X_i$ intersecting every orbit, with the property that the \emph{first return maps} to $X'_1$ and to $X'_2$ are topologically conjugate. Flow equivalence  is well-studied in symbolic dynamics and $C^{*}$-algebras.  By combining the Parry--Sullivan characterisation of flow equivalence with a general reconstruction result for groups of homeomorphisms by Ben Ami and Rubin \cite{BAR} (see Theorem~\ref{t-BAR}), we observe the following.
\begin{introthm} \label{t-isom}
For $i=1,2$, let $(X_i, \varphi_i)$ be a Stone system. Assume that $(X_i, \varphi_i)$ admits a dense orbit for $i=1,2$. Then the groups $\Tsf(\varphi_1)$ and $ \Tsf(\varphi_2)$ are isomorphic if and only if the system $(X_1, \varphi_1)$ is flow equivalent to $(X_2, \varphi_2)$ or to $(X_2, \varphi_2^{-1})$. 
\end{introthm}

In \cite{APP}, Aliste-Prieto and Petite observed that \cite{BAR} implies the analogous result for $\Hsf_0(\varphi)$.

Let us conclude this introduction with a question, which aims at clarifying a connection between the outer automorphism group $\Out(\Tsf(\varphi))$ and the \emph{mapping class group} of subshifts, extensively studied in  \cite{APP,Boyle,BCE,BC,SY}.

For simplicity, let us restrict to the case where $(X, \varphi)$ is a topologically transitive subshift. Note that by topological transitivity, every element $g\in \Hsf(\varphi)$ must simultaneously preserve or reverse the orientation of every $\Phi$-orbit, as this is determined by what it does on any dense $\Phi$-orbit. The \emph{complete mapping class group} of $(X, \varphi)$ is defined as the quotient $\MCG_{\pm}(\varphi)=\Hsf(\varphi)/\Hsf_0(\varphi)$~\footnote{The {mapping class group} of $(X, \varphi)$ is more usually defined as its index 2 subgroup $\MCG_+(\varphi)=\Hsf_+(\varphi)/\Hsf_0(\varphi)$ of mapping classes of elements that send each $\Phi$-orbit to a $\Phi$-orbit in an orientation-preserving way.}. This group is countable, see \cite[Corollary 2.3]{BC}.

The result of Ben Ami and Rubin \cite{BAR}  allows to identify the automorphism group $\Aut(\Tsf(\varphi))$ with the normaliser of $\Tsf(\varphi)$ in $\Hsf(\varphi)$. In particular, we have a homomorphism $\Out(\Tsf(\varphi))=\Aut(\Tsf(\varphi))/\Tsf(\varphi)\to \MCG_{\pm}(\varphi)$.  
\begin{ques} Let $(X,\varphi)$ be a topologically transitive subshift. Is the map $\Out(\Tsf(\varphi))\to \MCG_{\pm}(\varphi)$ a group isomorphism?
\end{ques}

Of course the answer might depend non-trivially on $(X, \varphi)$. The result of Brin \cite{Bri} can be reinterpreted by saying that for Thompson's group $T$ we have indeed an isomorphism $\Out(T)\to \MCG_{\pm}(\T)=\Z/2\Z$.

Finally this paper contains an appendix (Appendix \ref{s-polish}), in which we endow the group $\Hsf_0(\varphi)$ with a natural Polish group topology, and show that the group $\Tsf(\varphi)$ is dense in $\Hsf_0(\varphi)$ with respect to this topology. This will be useful for the discussion in Section \ref{s-subgroups}.

\section{Preliminaries on simplicity of groups of homeomorphisms}\label{ss-simple}

In this short section, we recall a well-known method to study normal subgroups  in groups of homeomorphisms and prove their simplicity.  Moreover, this  serves in part as a motivation for our construction.

 For every group $G$ acting by homeomorphisms on a topological space $Y$, and every open subset $U\subset Y$, the \emph{rigid stabiliser} $\rist(U)$ is defined as the subgroup consisting of elements that move only points in  $U$. The action of $G$ on $Y$ is said to be \emph{micro-supported} if $\rist(U)$ is non-trivial for every non-empty open subset $U$.

  The short argument giving the following lemma is classical and has been used in many proofs of simplicity.  With this formulation, it is exactly \cite[Lemma 4.1]{Nek-fp}. 
 
 \begin{lem}[(Double Commutator Lemma)]\label{l-epstein}
	Let $G$ be a group of homeomorphisms of a Hausdorff space $Y$, and $N\unlhd G$ be a non-trivial normal subgroup. Then there exists a non-empty open subset $U\subset Y$ such that the derived subgroup $\rist(U)'$ of the rigid stabilizer is contained in $N$. 
\end{lem}
\begin{proof}
Take a non-trivial element $g\in N$ and an open subset $U\subset Y$ such that $g(U)\cap U=\emptyset$. Then  for every two elements $h_1$ and $h_2\in \rist(U)$, the element $gh_1g^{-1}$ is supported in $g(U)$ and therefore commutes with $h_1$ and $h_2$. Using this,  one readily checks that the commutator $[h_1,h_2]=h_1h_2h_1^{-1}h_2^{-1}$ equals $[[h_1,g],h_2]$, and therefore belongs to $N$.
\end{proof}

Note that the lemma is valid without any assumptions on the action of $G$ on $Y$, but its conclusion is trivial unless the action is {micro-supported}.

We will use the following immediate consequence of the Double Commutator Lemma, which is a variant of Epstein's simplicity criterion \cite{Epstein}. We say that a group of homeomorphisms of a Hausdorff space $Y$ is \emph{fragmentable} if for every open cover $\mathcal{U}$ of $Y$, the group $G$ is generated by the derived subgroups $\mathsf{RiSt}(U)'$ for $U\in \mathcal{U}$.

\begin{prop}\label{p-simplicity}
Let $G$ be a  non-trivial group of homeomorphisms of a Hausdorff space $Y$. If $G$ is  fragmentable and acts minimally on $Y$, then   $G$  is simple. 
\end{prop}
\begin{proof}
Let $N\unlhd G$ be a non-trivial normal subgroup of $G$, and $V\subset Y$ be a non-empty open set given by the Double Commutator Lemma (Lemma \ref{l-epstein}). Observe that for every $g\in G$, the group $N$ contains $g\mathsf{RiSt}(V)'g^{-1}=\mathsf{RiSt}(g(V))'$. By minimality, the subsets $g(V)$ form an open cover of $Y$, and thus $N=G$.
\end{proof}

\begin{rem} \label{r-micro-supp-R}
Note that if $G$ is a \emph{finitely generated} simple group, then $G$ cannot admit any micro-supported action by homeomorphisms on any  locally compact, \emph{non-compact} Hausdorff space $Y$. To wit, assume by contradiction that this is the case. Let $N$ be the subgroup of $G$ generated by all rigid stabiliser $\rist(U)$ of relatively compact open subsets $U\subset Y$. Then $N$ is non-trivial and normal, so that by simplicity $N=G$. However, $N$ cannot be finitely generated unless $Y$ is compact.\end{rem}

\section{The group $\Tsf(\varphi)$ and its first properties}\label{s-group}

As mentioned earlier, we are assuming that $(X, \varphi)$ is a dynamical system consisting of a non-empty Stone space  $X$ and of a homeomorphism $\varphi\colon X\to X$. We denote by $Y^\varphi=\left(X\times \R\right)/\Z$ its suspension, and by $\Phi$ the suspension flow on $Y^\varphi$. We do not make any further assumptions on $(X, \varphi)$ unless  specified.

A subset $C\subset X$ is \emph{clopen} if it is both open and closed. Given a clopen subset $C$, its \emph{first return time} is the quantity 
\begin{equation}\label{e-smallest_return}\tau_C:=\min \{n\ge 1\colon \exists x\in C, \varphi^n(x)\in C\}\in \N\cup\{\infty\}.\end{equation}
Let $C\subset X$ be a clopen subset, and $J\subset \R$ be an open  interval  of length $|J|< \tau_C$ (in particular, this is always verified if $|J|<1$). For such a pair,  the inclusion of $C\times J$ into $X\times \R$ defines an injective map $\pi_{C, J}\colon C\times J\hookrightarrow Y^\varphi$, that we call a \emph{chart}. The image $U_{C \times J}=\pi_{C,J}(C\times J)$  will be called the \emph{domain} associated with $(C, J)$. 

The  domains satisfy the relation
\begin{equation} \label{e-chart} U_{C \times J}=U_{\varphi^n(C) \times (J-n)}, \quad n\in \Z. \end{equation}
Here $J-n$ stands for the translate $J-n=\{x-n\colon x\in J\}$.

We call an  interval $I\subset \R$ \emph{dyadic} if it is bounded and its endpoints are dyadic rationals, i.e.~rationals of the form $p/2^q$ for $p, q\in \Z$. A chart $\pi_{C, J}$ and a domain $U_{C\times J}$ are  \emph{dyadic} if $J$ is dyadic.
 Given two   intervals $I$ and $J\subset \R$, a \emph{piecewise linear (PL) homeomorphism} $f\colon I\to J$ is a homeomorphism which is piecewise affine, with finitely many discontinuity points for the derivatives (also called \emph{breakpoints}). We will further say that $f:I\to J$ is \emph{dyadic} if both $I$ and $J$ are dyadic and its breakpoints are dyadic rationals and the restriction of $f$  to every piece  is of the form $x\mapsto 2^mx +b$ for some $m\in \Z$ and $b$ dyadic rational. Recall that the group of all dyadic PL  homeomorphisms of $[0, 1]$ is \emph{Thompson's group} $F$, whereas the analogous definition on the circle gives \emph{Thompson's group} $T$. Both groups are finitely generated. For preliminaries on Thompson's groups we refer to \cite{CFP,Bi-St}.

\begin{dfn} \label{d-group}
We let  $\PL(\varphi)$ (resp.~$\Tsf(\varphi)$) be the group consisting of all elements $g\in \Hsf_0(\varphi)$ such that for every $y\in Y^\varphi$, there exist a dyadic chart $\pi_{C,J}$ with  $y\in U_{C\times J}$, and  a PL homeomorphism (resp.~a dyadic PL homeomorphism) $f\colon J\to f(J)\subset \R$ such that $g(U_{C\times J})=U_{C\times f(J)}$ and such that  the map $\pi_{C, f(J)}^{-1}\circ g\circ \pi_{C, J} \colon C\times J \to C\times f(J)$ is given in coordinates by  $(x, t)\mapsto (x, f(t))$.
\end{dfn}

\begin{rem} Clearly we have the inclusions $\Tsf(\varphi) \le \PL(\varphi) \le \Hsf_0(\varphi)$. 
 While the groups $\Hsf_0(\varphi)$ and $\PL(\varphi)$ are always uncountable, the group $\Tsf(\varphi)$ is countable as soon as $X$ is metrisable. In fact, in this case the set of clopen subsets of $X$ is countable, and so is the set of dyadic charts. Since every $g\in \Tsf(\varphi)$ is determined by a finite collection of dyadic charts and of dyadic $\PL$-homeomorphisms, the conclusion follows.  
\end{rem}

Let us clarify when these groups are left-orderable. The system $(X, \varphi)$ is said to be \emph{topologically free} if the set of infinite orbits is dense in $X$. By Baire Category Theorem, this is equivalent to the fact that for every $n\ge 1$, the set of points of period $n$ has empty interior.

\begin{prop}\label{p-lo}
	If $(X,\varphi)$ is topologically free,    then $\Hsf_0(\varphi)$ (and hence $\Tsf(\varphi)$) is left-orderable. If, moreover, $X$ is metrisable, then $\Hsf_0(\varphi)$ is isomorphic to a subgroup  of $\Homeo_+(\R)$.
	
\end{prop}

\begin{proof}
	Observe that if $(X,\varphi)$ has a dense collection of non-periodic orbits, then also $(Y^\varphi,\Phi)$ does.
	Let $(\mathcal O_\lambda)_{\lambda\in \Lambda}$ be a collection  of non-periodic $\Phi$-orbits, whose union is dense in $Y^\varphi$. Every orbit $\mathcal O_\lambda$ can be identified with the real line $\R$, and the restriction of the action of $\Hsf_0(\varphi)$ defines a morphism $\rho_\lambda:\Hsf_0(\varphi)\to \Homeo_+(\mathcal O_\lambda)$. Observe that density of $\bigcup \mathcal O_\lambda$ implies that the morphism \mbox{$(\rho_\lambda): \Hsf_0(\varphi)\to \prod \Homeo_+(\mathcal O_\lambda)$} is injective. Thus $\Hsf_0(\varphi)$ embeds in a product of left-orderable groups and is therefore  left-orderable. If moreover $X$ is metrisable, we can extract a countable  collection $(\mathcal{O}_n)_{n\in \N}$ from $(\mathcal{O}_\lambda)$ whose union remains dense. Identifying each $\mathcal{O}_n$ with the interval $(n, n+1)$, we obtain a faithful action of $\Hsf_0(\varphi)$ on $\R$.
\end{proof}

We will use the following fact.

\begin{lem}[(Minimality)]\label{l-minimal}
	For every $y\in Y^\varphi$, the $\Tsf(\varphi)$-orbit of $y$ is dense in the $\Phi$-orbit of $y$. In particular the $\Tsf(\varphi)$-action on $Y^\varphi$ is minimal if the system $(X, \varphi)$ is so.
\end{lem}

\begin{proof}
	The group $\Tsf(\varphi)$ contains the  translation $[x,s]\mapsto [x,s+r]$ for every dyadic rational $r$ (cf.~Proposition \ref{p-nonamenable}). That is, the group $\Tsf(\varphi)$ contains a dense set of times of the flow $\Phi$, so its action on every $\Phi$-orbit must be minimal.
	\end{proof}

\section{A family of subgroups of $\Tsf(\varphi)$}\label{s-subgroups}
In this section we introduce an (infinite) family of subgroups of $\Tsf(\varphi)$ isomorphic to Thompson's group $F$ (see Definition \ref{d-local-F}), which constitute a generating system for $\Tsf(\varphi)$, and derive some useful properties of this family.

For every open  dyadic interval $J\subset \R$, we will denote by $F_J$  the group of dyadic PL self-homeomorphisms of $\R$ which are the identity on the complement of $J$. Defined this way, $F_J$ is naturally a subgroup of $\Homeo_+(\R)$.  If  ${J}$ is non-empty, $F_J$ is always conjugate to Thompson's group $F=F_{(0,1)}$, a conjugacy being given by any  dyadic PL homeomorphism which maps $J$ onto $ (0,1)$.    Its derived subgroup $F'_J$ is simple and consists of elements whose \emph{closed} support is contained in~$J$. 
 More precisely, we have the following: for every open  dyadic interval $J\subset \R$, one has
\begin{equation}\label{e:derivedsub}
F'_J=\bigcup_{\overline{I}\subset {J}}F_I = \bigcup_{\overline{I}\subset {J}}F'_I,
\end{equation}
 where $I$ varies over open dyadic intervals whose closure is contained in $J$. We will  use repeatedly the following well-known properties of Thompson's group $F$:
\begin{lem}\label{l:conj_in_F}
For any  open dyadic intervals $I,J_1$ and $J_2$ such that $I\supset \overline{J_1},\overline{J_2}$, there exists an element $h\in F'_I$ such that $h(J_1)=J_2$ and thus $hF_{J_1}h^{-1}=F_{J_2}$. 
\end{lem}
\begin{proof}
See Cannon, Floyd and Parry \cite[Lemma 4.2]{CFP}.
\end{proof}

\begin{lem}[(Fragmentation Lemma for $F$)]\label{l-coverF}
Let $J$ be an open dyadic interval covered by a family $ \{J_i\}_{i\in \mathcal{L}}$ of open dyadic intervals $J_i\subset J$. Then the group $F'_{J}$ is contained in the group generated by $\bigcup_{i\in \mathcal{L}} F'_{J_i}$.
\end{lem}

\begin{proof} Although this is well-known we prefer to include a proof here. 
	Let us denote by $K$ the group generated by $\bigcup_{i\in \mathcal{L}} F'_{J_i}$. By \eqref{e:derivedsub}, it is enough to show that $F'_I\le K$ for every open dyadic interval such that  $\overline{I}\subset J$. Fix such an $I$ and extract from $\mathcal L$ a finite subcover $J_1,\ldots, J_n$ of $\overline{I}$ of minimal cardinality. We argue by induction on the number  $n$. If $n=1$, we have $F_I\le F_{J_1}\le K$. Assume that $n\ge 2$, and let $a_i$ be the leftmost point of $J_i$, for $i=1,\ldots,n$. These are pairwise distinct:  if $J_i$ and $J_j$ share an endpoint, we would have $J_i\subset J_j$ or $J_j\subset J_i$, contradicting the minimality of $n$. Without loss of generality we can assume  that $a_1<a_2<\ldots <a_n$, and that $J_i\cap J_{i+1}\neq \emptyset$ for $i=1,\ldots, n$. Let $a$ be the leftmost point of $I$ and choose $g\in F'_{J_1}$ such that $g(a)\in J_2$ (this exists by Lemma \ref{l:conj_in_F}). Then $g(\overline{I})\subset J_2\cup\cdots \cup J_n$, so that $gF'_{I}g^{-1}=F'_{g(I)}\le K$ by the induction hypothesis. It follows that $F'_I\le K$.   
\end{proof}
\begin{dfn}\label{d-local-F}
Every dyadic chart $\pi_{C, J}$ defines an embedding of the group $F_J$ into $\Tsf(\varphi)$ obtained by letting $F_J$ act on $C\times J$ with trivial action on $C$, and transporting this action on $U_{C\times J}$. We denote the image of  this embedding by 
$F_{C, J}\le \Tsf(\varphi)$.
\end{dfn}

\begin{prop}\label{p-inf-gen}
The group $\Tsf(\varphi)$ is generated by the derived subgroups $F'_{C, J}$, defined by all dyadic charts $\pi_{C, J}$. \end{prop}
\begin{rem}\label{r-inf-gen}
We postpone the proof of Proposition \ref{p-inf-gen} to the appendix (Appendix \ref{s-polish}), since it becomes cleaner after introducing a suitable  topology on the group $\Hsf_0(\varphi)$.  In the relevant special case where $(X, \varphi)$  is minimal, a more direct proof will appear in Section \ref{s-circle} (see Remark \ref{r-p-inf-gen-minimal}).  
We point out however that Proposition \ref{p-inf-gen} is orthogonal to all other results in this paper:  the reader  may as well take as definition that the group $\Tsf(\varphi)$ is the subgroup of $\Hsf_0(\varphi)$ generated by the subgroups $F_{C, J}\le \Hsf_0(\varphi)$, and continue reading. We decided to include it for completeness, since we believe that Definition \ref{d-group} is more natural.
  \end{rem}

We thank the referee for pointing out that, at this stage, Proposition \ref{p-inf-gen} already yields the following, which can be seen as a weak form of simplicity.

\begin{cor}\label{c-perfect}
	The group $\Tsf(\varphi)$ is perfect and admits no non-trivial finite quotient.
\end{cor}
\begin{proof}
	After Proposition \ref{p-inf-gen}, the group $\Tsf(\varphi)$ is generated by commutators, and thus is perfect.	
	Let $N\lhd \Tsf(\varphi)$ be a proper normal subgroup.
	After Proposition \ref{p-inf-gen}, there exists a dyadic chart $\pi_{C, J}$ such that $(\Tsf(\varphi)\setminus N)\cap F'_{C,J}\neq\emptyset$. As $F'_{C,J}\cong F'$ is simple, we must have $F'_{C,J}\cap N=\{\id\}$. This implies that the quotient $\Tsf(\varphi)/N$ is infinite.
\end{proof}

We conclude this section with  two lemmas on the subgroups $F'_{C, J}$, that will be repeatedly used later. 
\begin{lem}[(Intersection Lemma)]\label{l-intersection0} \label{l-intersection}
Let $\pi_{C, I}$ and $\pi_{D, J}$ be dyadic charts such that $C\cap D\neq \emptyset$ and ${I}\cap {J}\neq \emptyset$. Then the group $\langle F'_{C, I}, F'_{D, J} \rangle$ contains the subgroups $F'_{C\cap D, I\cap J}$, $F'_{C\cap D, I}$ and $F'_{C\setminus D,I}$.
\end{lem}
\begin{proof}
Let us first show that $F'_{C\cap D, I\cap J}\le \langle F'_{C, I}, F'_{D, J} \rangle$. Observe that for every dyadic intervals $I$ and $J\subset\R$, one has $F'_I\ge F'_{I\cap J}$. Thus
the group $F'_{C, I}$ contains $F'_{C, I\cap J}$ as subgroup, and similarly $F'_{D, J}$ contains $F'_{D, I\cap J}$. 
Therefore, it suffices to prove that $F'_{C\cap D, I\cap J}$ is contained in the group generated by $F'_{C, I\cap J}$ and $F'_{D, I\cap J}$. To simplify notations, we will assume $I=J=I\cap J$. Take $f\in F'_{C, I}$ and $g\in F'_{D, I}$, defined respectively by $f_0$  and $g_0\in F'_I$. Then the commutator $[f,g]$ is supported on $U_{(C\cap D)\times I}$, where it is defined on the flow direction by the commutator $[f_0,g_0]\in F'_I$. As commutators generate $[F'_I,F'_I]=F'_I$, we have the statement.

Let us now show that $F'_{C\cap D, I} \le  \langle F'_{C, I}, F'_{D, J} \rangle$. We already know that $\langle F'_{C, I}, F'_{D, J} \rangle$ contains $F'_{C\cap D, I\cap J}$.  Fix an open dyadic interval $L$ such that $\overline{L}\subset {I}$. By Lemma \ref{l:conj_in_F} we can find $h\in F'_I$ such that $h(L)\subset I\cap J$. Let $g\in F_{C, I}$ be the element  given in coordinates by $\id \times h$. Then $F'_{C\cap D,L}=g^{-1}F'_{C\cap D, h(L)}g$ is also contained in $\langle F'_{C, I}, F'_{D, J} \rangle$. Since $L$ was arbitrary, by \eqref{e:derivedsub} we have that $F'_{C\cap D, I}\le \langle F'_{C, I}, F'_{D, J} \rangle$.

Finally, we show that $F'_{C\setminus D,I} \le \langle F'_{C, I}, F'_{D, J} \rangle$. Take  an element $g\in F'_{C\setminus D, I}$, defined in the flow direction by a  dyadic PL  homeomorphism $f:I\to I$. Take the elements $g_1\in F'_{C,I}$ and $g_2\in F'_{C\cap D,I}$, defined by the same $f$. Then we have the equality $g=g_1g_2^{-1}$. As $g\in F'_{C\setminus D, I}$ was arbitrary, this proves that $F'_{C\setminus D, I}$  is also a  subgroup of $\langle F'_{C, I}, F'_{D, J} \rangle$.
\end{proof}

\begin{lem}[(Fragmentation Lemma)] \label{l-cover}
Assume that $C\times J$ is covered by a family $\{C_i\times J_i\}_{i\in \mathcal{L}}$, for clopen subsets $C_i\subset C$ and open dyadic intervals $J_i\subset J$. Then the group $F'_{C, J}$ is contained in the group generated by $\bigcup_{i\in \mathcal{L}} F'_{C_i, J_i}$.
\end{lem}

\begin{proof}
Let us denote by $H\le \Tsf(\varphi)$ the subgroup generated by  $\bigcup_{i\in \mathcal{L}} F'_{C_i, J_i}$, and let us show that $H=F'_{C, J}$. Assume first that $J_i=J$ for every $i$. By compactness we can find a finite cover $C_1,\ldots, C_r$ of $C$. Taking finitely many intersections and symmetric differences and using the Intersection Lemma (Lemma \ref{l-intersection}) we can replace it by a refinement $D_1,\ldots, D_s$ consisting of pairwise disjoint clopen subsets such that $F_{D_i, J}\le H$ for every $i$. Every element of $F'_{C,J}$ is a commuting product of elements in $F'_{D_i, J}$, whence the conclusion. Assume next that the intervals $J_i$ are arbitrary.
	\begin{claim}
	For any $t\in J$ there exists a subinterval $I\subset J$ such that $t\in I$ and  $F'_{C,I}\le H$.
	\end{claim}
	\begin{proof}[Proof of Claim]
	Using compactness, we can extract a finite subfamily $\mathcal F\subset \mathcal L$, such that $C\times \{t\}\subset \bigcup_{j\in\mathcal F}C_j\times J_j$ and such that $t\in J_j$ for every $j\in\mathcal F$. Let $I=\bigcap_{j\in\mathcal F} J_j$. Since $F_{C_j, I}\le F_{C_j, J_j}$, the previous case (applied to $J=I$) implies that $F'_{C, I}\le H$. 
	\end{proof}

By Lemma \ref{l-coverF}, we deduce that $F'_{C, J}\le H$.
\end{proof}

\section{Simplicity} \label{s-simplicity}

Theorem \ref{t-simple} will be an application of Proposition \ref{p-simplicity} (we refer to Section \ref{ss-simple} for the definition of fragmentable). 
\begin{prop}\label{p-fragmentable}
	Let $(X,\varphi)$ be a Stone system. Then the group $\Tsf(\varphi)$, as group of homeomorphisms of the suspension $Y^\varphi$, is fragmentable.
\end{prop}

\begin{proof}
	Let $\mathcal U$ be an open cover of $Y^\varphi$. Let $K$ be the subgroup of $\Tsf(\varphi)$ generated by $\mathsf{RiSt}(U)'$, for $U\in \mathcal{U}$, and let us show that $K=\Tsf(\varphi)$. By Proposition \ref{p-inf-gen}, it is enough to show that $F'_{C, I}\le K$ for every dyadic chart $\pi_{C, I}$. Since $U_{C\times I}$ is covered by $\mathcal{U}$ and dyadic domains form a basis of the topology, we can find a cover $C\times I=\bigcup_{i\in \Lambda} C_i \times I_i$, with $C_i\subset C$ and $I_i\subset I$, such that each dyadic domain $U_{C_i \times I_i}$ is contained in some element  $U_i\in \mathcal{U}$. Note that $F_{C_i, I_i}\le \rist(U_i)'$. By the Fragmentation Lemma (Lemma \ref{l-cover}), we have $F'_{C, I}\le \langle F_{  I_i, C_i}\rangle_{i\in \Lambda} \le \langle \rist(U_i)'\rangle_{i\in \Lambda}\le K$, as desired. 
\end{proof}

\begin{proof}[Proof of Theorem \ref{t-simple}]
Suppose that $(X,\varphi)$ is minimal. By Lemma \ref{l-minimal}, the action of $\Tsf(\varphi)$ on $Y^\varphi$ is minimal, and after Proposition~\ref{p-fragmentable}, the group $\Tsf(\varphi)$ is fragmentable. Hence we can apply Proposition \ref{p-simplicity} and obtain that $\Tsf(\varphi)$ is simple.
\end{proof}

\section{Finite generation}\label{s-fg}

In this section we prove Theorem \ref{t-fg}.
We first recall a  classical characterisation of subshifts in terms of existence of generating partitions. Let $(X, \varphi)$ be a Stone system.  A \emph{generating partition} for $(X, \varphi)$ is a clopen partition $X=C_1\sqcup \cdots \sqcup C_d$ such that the sets $\varphi^n(C_i)$, for  $n\in \Z$, and $i=1,\ldots, d$, separate points of $X$, or equivalently  generate the boolean algebra of clopen subsets of $X$. It is well-known that a generating partition exists if and only if $(X, \varphi)$ is conjugate to a subshift. Indeed, if $X=C_1\sqcup \cdots \sqcup C_d$ is generating, the map $X\to \{1,\ldots, d\}^\Z$ sending $x\in X$ to the unique sequence $(i_n)\in \{1,\ldots, d\}^\Z $ such that $\varphi^n(x)\in C_{i_n}$ provides such a conjugation. Conversely if $X\subset \{1,\ldots , d\}^\Z$ is a subshift, the subsets $C_i=\{(x_n)_{n\in\Z}\in X \colon  x_0=i\}$ provide a generating partition for $(X, \varphi)$.

In what follows we let $(X, \varphi)$ be a Stone system conjugate to a subshift.  Fix the dyadic intervals $I_0=(-1/4, 1/2)$ and $I_1=(1/4, 1+1/8)$. The only relevant properties of these two dyadic intervals will be that $|I_0|, |I_1|<1$, they intersect non-trivially and their union is a neighbourhood of  $[0, 1]$.

\begin{lem}\label{l-moving-intervals}Let $(X, \varphi)$ be a Stone system.
Let $K\le \Tsf(\varphi)$ be the group generated by $F_{X, I_0}$ and $F_{X, I_1}$. Let $C\subset X$ be a clopen subset and $I,J$ be open dyadic intervals of length $|I|,|J|<1$. Then there exists an element $k\in K$ such that $k(U_{C\times J})=U_{C\times I}$. 
\end{lem}

\begin{proof}
	Observe that $F_{X,I_i}=F_{\varphi^n(X),I_i-n}=F_{X,I_i-n}$ for every $i=0,1$ and $n\in \Z$, by the domain relations \eqref{e-chart}. Thus $K$ contains $F_{X, L}$ whenever $L$ is an interval obtained by translating  $I_0$ or $I_1$ by an integer.   Let $U$ be the smallest interval containing both $I$ and $J$. Without loss of generality we assume that the rightmost point of $U$ coincides with that of $J$. Choose a cover of $\overline{U}$ by a sequence  $L_1, \ldots ,  L_m$ of intervals, each of which is a translate  of $I_0$ or $I_1$ by an integer, such that $L_i\cap L_{i+1}\neq \emptyset$, for every $i=1,\ldots, m-1$. We argue by induction on $m$. If $m=1$, by Lemma \ref{l:conj_in_F} we can find an element $f\in F_{L_1}$ such that $f(I)=f(J)$. The element $k\in F_{X, L_1}$ which  acts as $f$ in the flow direction will then verify the conclusion.  If $m\ge 2$, we can chose an element $f\in F_{L_m}$ wich maps the leftmost point of $J$ into $L_{m-1}$, so that $I$ and $f(J)$ are in the interval covered by $L_1,\ldots, L_{m-1}$.  Letting $g\in F_{X, L_m}$ be the  element corresponding to $f$, we have $g(U_{C\times  J})=U_{C\times  f(J)}$, and we conclude by the induction hypothesis applied to  the intervals $I$ and $f(J)$.
\end{proof}

\begin{prop} \label{p-fin-gen}
Let $(X, \varphi)$ be a Stone system. Assume that $(X, \varphi)$ is conjugate to a subshift, and that $X=C_1\sqcup\cdots \sqcup C_d$ is a generating clopen partition of  $(X, \varphi)$. 
Then the group $\Tsf(\varphi)$ is generated by the subgroups $F_{X, I_0}$ and $F_{C_i, I_1}$, for $i=1,\ldots, d$.

\end{prop}
\begin{proof}
Let $H \le \Tsf(\varphi)$ be  the group generated by the subgroups in the statement, and let us show that $H=\Tsf(\varphi)$. 
We begin by observing that every element  $g\in F_{X, I_1}$ can be written as a commuting product $g=g_1\cdots g_d$ of elements $g_i\in F_{C_i, I_1}$, so that  $H$ contains the group $K:=\langle F_{X, I_0}, F_{X, I_1}\rangle$ as in the statement of Lemma \ref{l-moving-intervals}. Thus conjugating the subgroups $F_{C_i, I_1}$ by elements of $K$, we deduce that  $H$ contains $F_{C_i, J}$ for every $i=1,\ldots, d$ and   every dyadic interval $J$ such that $|J|<1$. By the domain relations \eqref{e-chart},  this is actually equivalent to say that $H$ contains  $F_{\varphi^n(C_i), J}$ for every $n\in \Z$, every $i=1,\ldots, d$, and every $J$ with $|J|<1$. Using now the Intersection Lemma (Lemma \ref{l-intersection}), we obtain that $F_{D, J} \le H$ whenever  $D\subset X$ is a clopen subset that can be written as a finite intersection  of the form $D=\bigcap_{j=1}^k \varphi^{n_j}(C_{I_j})$ with $n_j\in \Z$ and $i_j=1,\ldots, d$ (and $|J|<1$ is an open dyadic interval). Now the fact that the partition $C_1,\ldots, C_d$ is generating  means exactly that an arbitrary clopen set  $C\subset X$ can be written as a finite union $C=D_1\cup\cdots \cup D_\ell$ of sets of that form. So the Fragmentation Lemma (Lemma \ref{l-cover}) implies that $F_{C, J}\le H$ for every clopen subset $C\subset X$  and every open dyadic interval $|J|<1$. By Proposition \ref{p-inf-gen}, we deduce that $H=\Tsf(\varphi)$.
\end{proof}

\begin{proof}[Proof of Theorem \ref{t-fg}]
Since the subgroups $F_{X, I_0}$ and $F_{C_i, I_1}$, for every  $i=1,\ldots, d$, are all isomorphic to Thompson's group $F$, and the latter is finitely generated, Proposition \ref{p-fin-gen} implies Theorem~\ref{t-fg}.
\end{proof}
The converse of Theorem \ref{t-fg} is also true: namely if $(X, \varphi)$ is not conjugate to a subshift, then the group $\Tsf(\varphi)$ is never finitely generated.
\begin{prop} \label{p-not-fin-gen}
Let $(X, \varphi)$ be a Stone system which is not conjugate to a subshift. Then the group $\Tsf(\varphi)$ is not finitely generated. 
\end{prop}
\begin{proof}
The proof is a standard argument in the setting of topological full groups,  see the beginning of the proof of Theorem 5.4 in   Matui  \cite{Mat-simple}. We sketch how to adapt the argument to this setting, following an approach kindly suggested by the referee. Given a Stone system $(X, \varphi)$, the projection $X\times \R \to \R$ passes to the quotient to a map $\eta \colon Y^\varphi \to \R/\Z$. For $t\in \R/\Z$ let $X_t=\eta^{-1}(\{t\})$.  Now let $\mathcal{B}$ be a $\varphi$-invariant Boolean algebra of clopen subsets of $X$. We can translate $\mathcal{B}$ to a Boolean algebra $\mathcal{B}_t$ of $X_t$ for each $t$, which only depends on $t$ modulo $1$ by $\varphi$-invariance of $\mathcal{B}$. For every $g\in \Tsf(\varphi)$ and every $t\in \R/\Z$, there exists  a clopen partition $X_t=C_1 \sqcup \cdots \sqcup C_n$ with the property that for every $i$,     the set $C_i$ is entirely contained in  the domain $U_{C \times I}$ of a chart where $g$ is given in coordinates by $\id \times f$ for some PL homeomorphism $f\colon I \to f(I)$.   We say that $g$ is $\mathcal{B}$-\emph{admissible} if for every $t\in \R/\Z$ such a partition can be chosen with  $C_i\in \mathcal{B}_t, i=1,\ldots, r$.  The set of $\mathcal{B}$-admissible elements is a subgroup $\Tsf(\varphi, \mathcal{B})$ of $\Tsf(\varphi)$, and it is easy to see that it is a strict subgroup unless $\mathcal{B}$ separates points of $X$ (i.e.\ coincides with the Boolean algebra of all clopen subsets of $X$). Moreover  every finitely generated subgroup $G\le \Tsf(\varphi)$ is contained in $\Tsf(\varphi, \mathcal{B})$ for some $\mathcal{B}$ which is \emph{finitely generated} as a $\varphi$-invariant Boolean algebra  (that is, it is generated as a Boolean algebra  by a finite number of clopen subsets together with all their $\varphi$-translates). Thus if $\Tsf(\varphi)$ is itself finitely generated, we deduce that the Boolean algebra of all clopen subsets of $X$ is itself finitely generated as a  $\varphi$-invariant Boolean algebra, which implies that  $(X, \varphi)$ must be conjugate to a subshift.
\end{proof}

Note that the same argument provides the following statement, that shows that the study of arbitrary finitely generated subgroups of $\PL(\varphi)$ can essentially be reduced to the case where $(X, \varphi)$ is a subshift. We record it here as it will be useful in Section \ref{s-T}.
\begin{lem}\label{l-fg-subshift}
Let $(X, \varphi)$ be a Stone system, and let $G$ be a finitely generated subgroup of $\PL(\varphi)$. Then there exists a subshift  $(\bar{X}, \bar{\varphi})$ and a homomorphism $G\to \PL(\bar{\varphi})$ with abelian kernel. 
\end{lem}
\begin{proof} In the proof we keep the same setting and notations from the proof of Proposition \ref{p-not-fin-gen}. For a $\varphi$-invariant Boolean algebra $\mathcal{B}$, we have a subgroup $\PL(\varphi, \mathcal{B}) \le \PL(\varphi)$ which is defined analogously to $\Tsf(\varphi, \mathcal{B})$. 

Fix a Stone system $(X, \varphi)$ and write $Y:= Y^\varphi$.
Let $G\le \PL(\varphi)$ be a finitely generated subgroup. Then we have  $G \le \PL(\varphi, \mathcal{B})$ for some finitely generated $\varphi$-invariant Boolean algebra $\mathcal{B}$. By Stone duality, there exists a Stone system $(\bar{X}, \bar{\varphi})$ and a factor map $p\colon X\to \bar{X}$ from $(X, \varphi)$ to $(\bar{X}, \bar{\varphi})$ such that  $\mathcal{B}$ coincides with the Boolean algebra of preimages of clopen subsets of $\bar{X}$ via $p$. The system $(\bar{X}, \bar{\varphi})$ is conjugate to a subshift (using that $\mathcal B$ is finitely generated). Set $\bar{Y}:= Y^{\bar{\varphi}}$, denote by  $\bar{\eta} \colon \bar{Y} \to \R/\Z$ the associated projection to the circle, and $\bar{X}_t:= \bar{\eta}^{-1}(t)$ its fibers for $t\in \R/\Z$. The map $p\colon X\to \bar X$ then defines a map $\hat{p} \colon Y \to \bar{Y}$, which in restriction to each fiber $X_t$ coincides with the quotient map $X_t \to \bar{X}_t$ associated to the Boolean algebra $\mathcal{B}_t$.  Every element $g\in \PL(\varphi, \mathcal{B})$ (and hence every $g\in G$) descends via the map $\hat{p} \colon Y \to \bar{Y}$ to an element $\bar{g}\in \PL(\bar{\varphi})$, providing the desired homomorphism $G\to \PL(\bar{\varphi})$. To see that its kernel must be abelian, let  $g\in G$ be such that $\bar{g}=\id$.  Since the map $\hat{p}$ sends $X_t$ to $\bar{X}_t$ for every $t \in \R/\Z$ and $\bar{g}$ acts trivially on $\overline{Y}$, we have that $g(X_t)=X_t$ for every $t$. This implies that in restriction to each $\Phi$-orbit $g$ must coincide with a translation by an integer time of the flow $\Phi$. Thus the kernel of $G \to \PL(\varphi)$ acts on each $\Phi$-orbit as a cyclic group of integer translations. This implies that  it embeds in a product of cyclic groups, showing that it is abelian. 
\end{proof}

\section{Actions on the circle}\label{s-circle}

In this section we study actions on the circle of $\Tsf(\varphi)$ and prove Theorem \ref{t-circle}. We begin with a lemma. 

\begin{lem}\label{l-compactness-prod}
Let $(X, \varphi)$ be a Stone system. 
Let $g\in \Tsf(\varphi)$ be an element whose closed support is contained  in a domain $U_{C\times I}$. Then there exists a partition $C=C_1\sqcup \cdots \sqcup C_r$ of $C$ into clopen subsets such that $g$ can be written as a commuting product $g=g_1\cdots g_r$ with $g_j\in F_{C_j, I}$ for $j=1,\ldots, r$.
\end{lem}
\begin{proof}

Fix $x\in C$. The element $g$ must preserve the fiber $\pi_{C\times I}(\{x\}\times I)$ and act on it by an element $f_x\in F'_I$ (because its closed support is contained in $I$). Furthermore, for every $y=\pi_{C,I}(x, t)\in U_{C\times I}$ there exists a clopen neighbourhood $C_y\subset C$ of $x$ and an interval $J_y\subset I$ containing $t$, such that for every $x$ and $x'\in C_y$ we have  $f_x|_{J_y}=f_{x'}|_{J_y}$. By compactness  choose $y_1,\ldots, y_\ell\in\pi_{C, I}(\{x\}\times I)$ such that $ J_{y_1},\ldots, J_{y_\ell}$ cover the support of $f_x$ and set $C_x=\bigcap_{i=1}^\ell C_{y_i}$. The clopen subset $C_x\subset C$ verifies $f_x=f_{x'}$ for every $x'\in C_x$. Since $x$ was arbitrary, we can cover $X$ with finitely many clopen subsets with this property, and refine this cover to obtain a partition $C=C_1\sqcup\cdots \sqcup C_r$ which verifies the desired conclusion.
\end{proof}

We will use the following notation. Given a group  $G$ acting by homeomorphism on a topological space  $Z$, and given a point $z\in Z$ we denote by $G_z$ the stabiliser of $z$ in $G$, and by $G^0_z$ the \emph{germ-stabiliser}, namely the subgroup of $G_z$ consisting of elements that fix pointwise a neighbourhood of $z$. The next two lemmas analyse the germ-stabilisers for the action of $\Tsf(\varphi)$ on $Y^\varphi$.

\begin{lem}[(First Return Decomposition)]\label{l-stab-prod}
Let $(X,\varphi)$ be a minimal Stone system, and fix $y\in Y^\varphi$.  For every finite subset  $P\subset \Tsf(\varphi)_y^0$, there exist disjoint dyadic domains $U_{C_1\times I_1},\ldots, U_{C_r\times I_r}$ not containing $y$ such that $P$ is contained in the direct product $F'_{ C_1, I_1}\times \cdots \times F'_{C_r, I_r} \le \Tsf(\varphi)_y^0$.

In particular, the germ-stabilizer  $\Tsf(\varphi)_y^0$ is perfect, has no non-abelian
free subgroup, and no non-trivial finite quotient.
  \end{lem}
\begin{proof}
Let $P\subset \Tsf(\varphi)_y^0$ be any finite subset. Choose a dyadic domain $U_{C\times J}\ni y$ with $J=(a, b)$ and $|J|<1/2$ such that every element $g\in P$ fixes pointwise a neighbourhood of the closure $\overline{U_{C\times J}}$. Write $\partial \overline{U_{C\times J}}=C_a\sqcup C_b$, where $C_a$ corresponds to $C\times\{a\}$ and $C_b$ to $C\times \{b\}$. By minimality, the  $\Phi$-orbit of every point in $C_b$ must eventually return to $U_{C\times J}$, and it must enter through a point in $C_a$. For $x\in C$, let $\tau(x)=\inf\{t>0\colon \Phi^{b+t}([x,0])\in C_a\}$. Note that $\tau(x)$ is of the form $n_x-(b-a)$ for some positive integer $n_x$ (equal to the first return time of $x$ to $C$ under $\varphi$).  The function $\tau\colon C\to \N-(b-a)$ is locally constant (hence continuous and bounded). Let $C=C_1\sqcup\cdots \sqcup C_r$ be the partition of $C$ into level sets  of $\tau$, with $\tau=\tau_j$ on $C_j$. By construction, the interval $I_j=(b, b+\tau_j)$ is such that the chart $\pi_{C_j,I_j}$ is defined (injective), and the domains $U_{{C_j}\times I_1}, \ldots, U_{{C_r}\times I_r}$ form a partition of $Y^\varphi\setminus \overline{U_{C\times J}}$. Every $g\in P$ must preserve every domain $U_{{C_j}\times I_j}$ and is the commuting product of elements $g_j$  supported in $U_{{C_j}\times I_j}$. By Lemma~\ref{l-compactness-prod}, upon refining the partition $C=\bigsqcup {C_j}$ we can assume that $g_j\in F'_{{C_j}\times I_j}$ for all $j$ and all $g\in P$.  After doing so,  $P$ is contained in   the subgroup of $\Tsf(\varphi)_y^0$ generated by the commuting subgroups $F'_{C_j\times I_j}\simeq F'$. 

As a consequence, similarly to the proof of Corollary \ref{c-perfect}, we deduce that the group $\Tsf(\varphi)_y^0$ is perfect and has no non-trivial finite quotient. Moreover, the group $F'$ does not contain non-abelian free subgroups \cite{Br-Sq}, so neither does  $\Tsf(\varphi)_y^0$.
\end{proof}

\begin{lem}\label{l-stab-generate}
Let $(X, \varphi)$ be a Stone system without finite orbits. For every two distinct points $y$ and $z\in Y^\varphi$,  we have \[\Tsf(\varphi)=\left(\Tsf(\varphi)_y^0\Tsf(\varphi)_z^0\right)\cup\left(\Tsf(\varphi)_z^0\Tsf(\varphi)_y^0\right).\] In particular $\Tsf(\varphi)$ is generated by  $\Tsf(\varphi)_y^0$ and $\Tsf(\varphi)_z^0$. 
\end{lem}
\begin{proof}

Let $g\in \Tsf(\varphi)$. The points $y$ and $g(y)$ lie on the same $\Phi$-orbit,  and delimit a unique interval on this orbit (possibly reduced to a point).

Assume at first that $z$ does not lie on this interval (in particular, this holds if $y$ and $z$ lie on different $\Phi$-orbits, or if $y=g(y)$). 
Choose a sufficiently small domain $U_{C\times I}\ni y$ such that $g$ is given in coordinates by $\id \times f$ for  a PL homeomorphism $f\colon I \to f(I)$, and write $y=[x, t]$, with $x\in C$ and $t\in I$. Let $J\subset \R$ be an open dyadic interval which contains both $t$ and $f(t)$. Since we assume that $\varphi$ has  no periodic orbits, we can find a clopen neighbourhood $D$ such that $x\in D\subset C$, and whose first return time is greater than $|J|$ (this was defined in \eqref{e-smallest_return}). With this choice, the domain $U_{D\times J}$ is well-defined and $y\in U_{D\times J}$. Furthermore, since $z$ does not lie on the arc between $y$ and $g(y)$, we can shrink this domain so that it avoids a neighbourhood of $z$. Now let $k\in F'_J$ be an element that coincides with $f^{-1}$ on a neighbourhood of $f(t)$, and $h=\id\times k\in F'_{D, J}$. Then $h\in \Tsf(\varphi)_z^0$ and $hg\in \Tsf(\varphi)_y^0$, so that $g=h^{-1} hg\in \Tsf(\varphi)_z^0\Tsf(\varphi)_y^0$.

We now consider the case where $z$  lies on the interval between $y$ and $g(y)$ on the corresponding $\Phi$-orbit. Let $\mathcal{O}$ be this orbit. Assume that $y<g(y)$ with respect to the natural orientation on $\mathcal{O}$, so that $y<z\le g(y)$. Since $g$ acts on $\mathcal{O}$ in an orientation-preserving way, this implies that  $g(y)<g(z)$, and thus $y$ does not lie on the interval between $z$ and $g(z)$. Thus we can exchange the roles of $z$ and $y$ and repeat the same argument as above to see that $g\in   \Tsf(\varphi)_y^0\Tsf(\varphi)_z^0$. The case $g(y)<y$ is treated similarly.
\end{proof}

\begin{rem}\label{r-p-inf-gen-minimal}
The previous two lemmas imply Proposition \ref{p-inf-gen} in the special case where $(X, \varphi)$ is an infinite minimal Stone system (note that the proposition has not been used in  their proofs). Indeed, in this case Lemma \ref{l-stab-prod}  implies that every germ-stabiliser $\Tsf(\varphi)_y^0$ is contained in the subgroup generated by the subgroups $F_{C, I}$, when $\pi_{C, I}$ varies over dyadic charts. By Lemma \ref{l-stab-generate}, the germs stabilisers generate $\Tsf(\varphi)$, and Proposition \ref{p-inf-gen} follows. 
\end{rem}

\begin{lem}\label{l-stab-fix}
Let $(X,\varphi)$ be a minimal Stone system. For every $y\in Y^\varphi$, every action of the group $\Tsf(\varphi)_y^0$ on the circle has a fixed point.

\end{lem}
\begin{proof}
By Lemma \ref{l-stab-prod}, the group $\Tsf(\varphi)_y^0$ does not contain non-abelian free subgroups. Thus, a result of Margulis \cite{Marg} implies that every action of the group $\Tsf(\varphi)_y^0$ on the circle preserves a Borel probability measure~$\mu$.
Therefore, the rotation number $g\in \Tsf(\varphi)_y^0\mapsto \mathsf{rot}(g)=\mu[\vartheta,g(\vartheta))\pmod{\Z}$ defines a group homomorphism to $\R/\Z$ (which does not depend on the choice of $\vartheta\in \Sbb^1$).
By Lemma~\ref{l-stab-prod}, the only possibility is that the image of $\mathsf{rot}$ is~$\{0\}$. Therefore every $g\in \Tsf(\varphi)_y^0$ fixes a point, and $\mu$ must be supported on fixed points of $g$, for every $g\in \Tsf(\varphi)_y^0$. Thus  there exists a global fixed point (any point in the support of~$\mu$).
\end{proof}

To conclude the proof of Theorem \ref{t-circle} we will work in the \emph{Chabauty space} of subgroups of $\Tsf(\varphi)$, by adapting arguments from \cite[\S4.1.4]{LB-MB-subdyn}. 
For every group $G$, we denote by $\sub(G)$ the set of subgroups of $G$. The \emph{Chabauty topology} on $\sub(G)$ is  the topology induced  by the inclusion of $\sub(G)$ into the set $\{0, 1\}^G$ of all subsets of $G$, endowed with the product topology, and this makes $\sub(G)$ a compact space.  Given a compact space $Z$, a map $H:Z\to \sub(G)$ is said to be upper (respectively lower) semi-continuous if for every net $(z_i)_i$ converging to a limit $z\in Z$, every cluster point $K$ of $\left (H({z_i})\right )_i$ in $\sub(G)$ verifies $K\le H(z)$ (respectively $H(z)\le K$).  If $G$ acts on $Z$ by homeomorphisms, the stabiliser map $z\mapsto G_z$ is always upper semi-continuous, while the germ-stabiliser map $z\mapsto G^0_z$ is always lower semi-continuous.

\begin{proof}[Proof of Theorem \ref{t-circle}] Let $(X, \varphi)$ be a minimal Stone system, with $X$ infinite. 
Consider an action of the group $\Tsf(\varphi)$ on $\Sbb^1$.  If it has a finite orbit, by simplicity of the group $\Tsf(\varphi)$ (Theorem \ref{t-simple}, although Corollary \ref{c-perfect} is enough) we deduce that it has a fixed point, and the conclusion follows. So we assume that the action does not have finite orbits. By general properties of groups of circle homeomorphisms,  every group action on the circle without finite orbit is either minimal or semi-conjugate to a minimal action (see \cite[Propositions 5.6 and 5.8]{Ghy-ens}). We can therefore assume that the action of $\Tsf(\varphi)$ on $\Sbb^1$ is minimal and show that this leads to a contradiction.
\begin{claim}
Assume that $\Tsf(\varphi)$ acts minimally on $\Sbb^1$. For every $\vartheta\in \Sbb^1$ there exists a unique point $q(\vartheta)\in Y^\varphi$ such that $\Tsf(\varphi)^0_{q(\vartheta)}\le \Tsf(\varphi)_\vartheta$. Moreover, the map $q\colon \Sbb^1\to Y^\varphi$ defined in this way is equivariant and continuous. 
\end{claim}

Before proving the claim, let us observe that it leads to the desired contradiction. Indeed the image $q(\Sbb^1)\subset Y^\varphi$ would be a compact (hence closed) $\Tsf(\varphi)$-invariant subset of $Y^\varphi$, which by minimality (Lemma~\ref{l-minimal}) would be equal to the whole space $Y^\varphi$. But  the space $Y^\varphi$ is not path-connected, and therefore cannot be a continuous image of $\Sbb^1$, reaching a contradiction.
We now prove the claim.
\begin{proof}[Proof of Claim]
Let us first observe that if the point $q(\vartheta)$ as in the statement exists, it must be unique. Indeed if $y\neq y'\in Y^\varphi$ are such that $\Tsf(\varphi)^0_y$ and $\Tsf(\varphi)_{y'}^0$ fix the point $\vartheta$, then by Lemma~\ref{l-stab-generate} we deduce that the whole group $\Tsf(\varphi)$ fixes $\vartheta$, contradicting minimality of the action on $\Sbb^1$. 

Let us now show the existence of the point $q(\vartheta)$. To this end fix $y_0\in Y^\varphi$, and let $\vartheta_0\in \Sbb^1$ be a point fixed by $\Tsf(\varphi)^0_{y_0}$ (which exists by Lemma \ref{l-stab-fix}), so that $\Tsf(\varphi)^0_{y_0}\le \Tsf(\varphi)_{\vartheta_0}$. Let $\vartheta \in \Sbb^1$ be an arbitrary point. By minimality, and upon extracting a subnet, we can choose a net $\left (g_i\right )_i\subset \Tsf(\varphi)$ such that: 
\begin{itemize}
	\item $(g_i(\vartheta_0))_i\subset \T$ converges to $\vartheta$,
	\item $\left (g_i(y_0)\right )_i\subset Y^\varphi$ converges to a limit $y\in Y^\varphi$,
	\item the nets $\left (\Tsf(\varphi)_{g_i(y_0)}^0\right )_i$ and $\left (\Tsf(\varphi)_{g_i(\vartheta_0)}\right )_i$ both converge   in $\sub(\Tsf(\varphi))$ to  limit subgroups denoted respectively by $H$ and $K$.
\end{itemize}
Since
\[\Tsf(\varphi)_{g_i(y_0)}^0= g_i\Tsf(\varphi)_{y_0}^0g_i^{-1}\le  g_i\Tsf(\varphi)_{\vartheta_0}g_i^{-1}=\Tsf(\varphi)_{g_i(\vartheta_0)},\]
we have $H\le K$, and by lower semi-continuity of the germ-stabiliser map and upper semi-continuity of the stabiliser  map we obtain $\Tsf(\varphi)_y^0\le H\le K\le \Tsf(\varphi)_\vartheta$. Thus $q(\vartheta):=y$ verifies the desired conclusion.

Equivariance of the map $q$ is clear. Let us show that it is continuous. Assume that $\left (\vartheta_i\right )_i\subset \Sbb^1$ is a net converging to  a limit $\vartheta\in \Sbb^1$. Let $y\in Y^\varphi$ be a cluster point of the net $\left (q(\vartheta_i)\right )_i$.  We have $\Tsf(\varphi)_{q(\vartheta_i)}^0\le \Tsf(\varphi)_{\vartheta_i}$, and using semi-continuity exactly as in the proof of existence of $q(\vartheta)$ we obtain that $\Tsf(\varphi)_y^0\le \Tsf(\varphi)_\vartheta$. By uniqueness of $q(\vartheta)$ we must have $y=q(\vartheta)$. Since~$y$ was an arbitrary cluster point of $\left (q(\vartheta_i)\right )_i$, we deduce that the net $\left(q(\vartheta_i)\right)_i$ converges to $q(\vartheta)$, showing continuity of the the map $q$.
\end{proof}
This concludes the proof of the theorem.
\end{proof}

\section{Central extensions of Thompson's group $T$}\label{s-LO}

In this section we explain that the group $\Tsf(\varphi)$ always contains subgroups isomorphic to a central extension of Thompson's group $T$ acting on the circle. In particular it always contains free subgroups, and is therefore non-amenable.

Let  $\widetilde{T}\le \Homeo_+(\R)$ be the natural lift of Thompson's group $T$ to the real line under the covering map $\R\to \R/\Z=\T$.   Explicitly, $\widetilde{T}$ consists of all   homeomorphisms of $\R$ which commute with integer translations and  are PL dyadic  with a (possibly infinite) discrete closed set of breakpoints for the derivative. It is a non-trivial central extension of Thompson's group $T$:
\[
0\to\Z\to\widetilde T\to T\to 1.
\]
   For each $k\ge 1$, let $T^{(k)}$ be the quotient of $\widetilde{T}$ by the subgroup consisting of translations by a multiple of $k$. Note that $T^{(k)}$ identifies with the lift of $T$ under  the $k$-fold self-covering of the circle $\T\to \T, x\mapsto kx$. 
   We have the following.
 
\begin{prop}\label{p-nonamenable} Let $(X, \varphi)$ be a Stone system. 
If $\varphi$ has infinite order, then $\Tsf(\varphi)$ contains a subgroup isomorphic to the group  $\widetilde T$. If $\varphi$ has  order  $k\ge 1$, the group $\Tsf(\varphi)$ contains a subgroup isomorphic to $T^{(k)}$. 
\end{prop}

\begin{proof}Let $\widetilde{T}$ act on $X\times \R$ diagonally with trivial action on $X$ and by its natural action on $\R$. This action commutes with the diagonal $\Z$-action on $X\times \R$ given by $n\cdot (x, t)\mapsto (\varphi^n(x), t-n)$, and thus passes to the quotient to an action on the suspension  $Y^\varphi$, which gives rise to a homomorphism $\widetilde{T}\to \Tsf(\varphi)$. Moreover, given an orbit $\cO\subset Y^\varphi$ of the suspension flow $\Phi$, the action of $\widetilde{T}$ on $\cO$ is faithful  unless $\cO$ is periodic, and in this case its kernel consists exactly of the subgroup of translations $k\Z$, where $k$ is the period of $\cO$. It follows  that if $\varphi$ is periodic of period $k$, the homomorphism $\widetilde{T}\to \PL(\varphi)$ factors through an injective homomorphism $T^{(k)}\to \PL(\varphi)$. Otherwise  $\varphi$ admits infinite orbits or finite orbits of arbitrarily large period, so that the homomorphism $\widetilde{T}\to \PL(\varphi)$ is injective.
\end{proof}
\begin{cor} \label{c-free-subgroups}
For every Stone system $(X, \varphi)$, the group $\Tsf(\varphi)$ contains free subgroups. 
\end{cor}
\begin{proof}
The group $T$ contains free subgroups, therefore so do its  extensions $\widetilde{T}$ and $T^{(k)}$ for  $k\ge 1$. We conclude by Proposition \ref{p-nonamenable}.
\end{proof}

We next prove Theorem~\ref{t:tilde-T}, which states that every action of the central extension $\widetilde T$ on $\R$, without global fixed point, is semi-conjugate to the standard action.
Before this, let us clarify what is a semi-conjugacy in this setting (in particular, for us a semi-conjugacy can reverse the orientation, cf.~the definitions in \cite{KKM}):

\begin{dfn}
	Let $\rho_\nu:G\to \Homeo_+(\R)$, $\nu=1,2$, be two representations. They are \emph{semi-conjugate} if the following holds: there exists a weakly monotone proper map $h:\R\to \R$ 
	such that
	\[
	h\,\rho_1(g)=\rho_2(g)\,h,\quad\text{for any }g\in G.
	\]
\end{dfn}

We also need the analogue notion for actions on the circle.
\begin{dfn}
	Let $\rho_\nu:G\to \Homeo_+(\T)$, $\nu=1,2$, be two representations. They are \emph{semi-conjugate} if the following holds: there exist
	\begin{itemize}
		\item a weakly monotone map $h:\R\to \R$ either commuting or anti-commuting with the integer translations (this depends on the monotonicity of $h$) and
		\item two corresponding central lifts $\widehat{\rho}_\nu:\widehat{G}\to \Homeo_{\Z}(\R)$  to homeomorphisms of the real line commuting with integer translations,
	\end{itemize}
	such that
	\[
	h\,\widehat{\rho}_1(\widehat{g})=\widehat{\rho}_2(\widehat{g})\,h,\quad\text{for any }\widehat{g}\in \widehat{G}.
	\]
\end{dfn}

\begin{rem}
Both for actions on the real line and on the circle, if a representation $\rho_1$ is semi-conjugate to a \emph{minimal} representation $\rho_2$, then the map $h$ realizing the semi-conjugacy is automatically continuous and surjective. (This is forced by the fact that the image of $h$ must be dense, together with the monotonicity of $h$). Here we will only use the notion of semi-conjugacy in this case. 
\end{rem}

The following result appears in \cite{Ghy-ens, LB-MB-subdyn}:

\begin{thm}\label{t-TonS1}
	Every non-trivial action of Thompson's group $T$ on the circle is semi-conjugate to the standard action.
\end{thm}

Let us observe that it implies the following:
\begin{thm}\label{t:tilde-T}
	Let $\widetilde T$ denote the central extension of Thompson's $T$. Every action of $\widetilde T$ by  homeomorphisms of the real line, without global fixed points, is semi-conjugate to the action lifting the action of $T$ on the circle.
\end{thm}

\begin{proof}
	Let $\rho:\widetilde T\to\Homeo_+(\R)$ be an action without global fixed points.
	Let us denote by $\gamma\in \widetilde T$ the central element corresponding to the translation $t\mapsto t+1$ in the standard action, so that $\langle \gamma \rangle$ is the center of $\tilde{T}$. 
	
	\begin{claim}
		The homeomorphism $\rho(\gamma)$ has no fixed point.
	\end{claim}
	
	\begin{proof}[Proof of Claim]
		Assume for contradiction that $Z:=\Fix(\rho(\gamma))\neq \emptyset$. As $\gamma$ is central, the subset $Z\subset \R$ is $\rho(\widetilde T)$-invariant, and the corresponding action of $\widetilde T$ on $Z$ passes to the quotient to an action of $T=\widetilde  T/ \langle \gamma \rangle$ which preserves the order of points of $Z\subset \R$.  As $T$ is simple and contains torsion elements, this implies that the induced action on $Z$ is trivial, that is, $\Fix(\rho(\gamma))=\Fix(\rho(\widetilde T))$. Thus the action has a global fixed point, giving the desired contradiction.
	\end{proof}
	Now after conjugating $\rho$ by an element of $\Homeo_+(\R)$  we can assume that $\rho(\gamma)=\gamma^{\pm 1}$, so that $\rho(\langle \gamma \rangle)= \langle \gamma \rangle$. Then the action $\rho:\widetilde T\to\Homeo_+(\R)$ descends to a non-trivial action \mbox{$\bar{\rho}:T\to \Homeo_+(\T)$}, where $\T=\R/\langle \gamma \rangle$, which by Theorem~\ref{t-TonS1} is semi-conjugate to the standard action. Thus, upon replacing again $\rho$ with a semi-conjugate representation, we can assume that $\rho$ is a homomorphism $\rho\colon \widetilde{T} \to \widetilde{T} \le \Homeo_+(\R)$ such that $\tilde{\rho}(\langle \gamma \rangle)= \langle \gamma \rangle$, and which passes to the quotient to the identity map $T \to T$. This implies that for every $g\in \widetilde T$ we have  $\rho(g)= g \gamma^{n_g}$ for some integer $n_g\in \Z$. 	Using that $\gamma$ is central one readily checks that the map $g\mapsto n_g$ is a homomorphism  $\widetilde{T} \to \Z$, and since $\widetilde T$ is perfect, we must have $n_g=0$ for all $g \in \widetilde T$. Thus $\rho$ is the identity.
\end{proof}

\section{Approximation by subshifts of finite type and lack of Kazhdan subgroups} \label{s-T}

In this  section we prove the non-existence of infinite subgroups with Kazhdan's property $(T)$ in the group $\PL(\varphi)$.  As another consequence of the arguments used, we also observe that for every infinite minimal subshift $(X, \varphi)$ the group $\Tsf(\varphi)$ is not finitely presentable.

Let us set up some extra notation. Given a Stone system $(X,\varphi)$ and a closed subset $Z\subset X$ which is $\varphi$-invariant, we denote by $\varphi\vert_Z$ the restriction of $\varphi$ to $Z$ and by $Y^\varphi_Z\subset Y^\varphi$ the suspension of $(Z, \varphi|_Z)$, which is a closed $\Phi$-invariant subset of $Y^\varphi$.

We begin with two  lemmas that  hold for any Stone system $(X, \varphi)$.

\begin{lem}\label{l:surjection}
Let $(X,\varphi)$ be a Stone system and let $Z\subset X$ be a $\varphi$-invariant closed subset. We let $G$ be either $\PL(\varphi)$ or $\Tsf(\varphi)$ and $G_Z=\PL(\varphi|_Z)$ or $\Tsf(\varphi|_Z)$ accordingly.
The restriction of the action of $G$ to $Y^\varphi_Z$ defines an epimorphism $G\to G_Z$, whose kernel is infinite unless $Z=X$. \end{lem}
\begin{proof}
Consider the case $G=\Tsf(\varphi)$.
It is clear that the restriction of the action defines a group homomorphism $\Tsf(\varphi)\to \Tsf(\varphi|_Z)$, and we need to show that it is surjective. Let $C\subset Z$ be  a clopen subset and $J$ be a dyadic interval of length $|J|< 1$. Choosing a clopen subset $D\subset X$ such that $C=D\cap Z$, we see that the subgroup $F'_{C,J}\le \Tsf(\varphi|_Z)$ is the image of the subgroup $F'_{D,J}\le \Tsf(\varphi)$. Therefore the image of $\Tsf(\varphi)$ in $\Tsf(\varphi|_Z)$ contains a generating set of $\Tsf(\varphi|_Z)$.  Finally, taking a clopen subset $C \subset (X\setminus Z)$ and any interval $I$ with $|I|< 1$, we see that $F_{C, I}\le \Tsf(\varphi)$ acts trivially on $Y^\varphi_Z$ and therefore the kernel of the map $\Tsf(\varphi)\to \Tsf(\varphi|_Z)$ is infinite. 

 The proof is identical for $G=\PL(\varphi)$ (one needs to replace the groups $F'_{C,I}$ with groups defined similarly for $\PL(\varphi)$).
\end{proof}
We point out that this implies that minimality is a necessary condition in Theorem \ref{t-simple}.
\begin{cor} \label{c-not-simple}
If the system $(X, \varphi)$ is not minimal, the group $\Tsf(\varphi)$ admits non-trivial proper quotients.
\end{cor}

\begin{lem} \label{l-limit} Let $(X, \varphi)$ be a Stone system. 
Consider a nested decreasing sequence of $\varphi$-invariant closed subsets $\left (Z_n\right )_n$ of $X$ and let $Z=\bigcap Z_n$.  We let $G_Z$ be either $\PL(\varphi|_Z)$ or $\Tsf(\varphi|_Z)$ and $G_{Z_n}=\PL(\varphi|_{Z_n})$ or $\Tsf(\varphi|_{Z_n})$ accordingly.
Then $G_Z$ is isomorphic to the direct limit $\varinjlim G_{Z_n}$ with respect to the restriction maps.
\end{lem}

\begin{proof}
Let $G$ be either $\PL(\varphi)$ or $\Tsf(\varphi)$, depending on the choice of $G_Z$.
Fix $g\in G$ which projects trivially to $G_Z$. By Lemma \ref{l:surjection}, this implies that every point  $y\in Y^\varphi_Z$ is contained in a  domain $U_{C\times I}$ on which $g$ acts trivially. Thus $g$ acts trivially on a neighbourhood $U$ of $Y^\varphi_Z$ in $Y^\varphi$. By compactness, for $m$ large enough, we have $Y^\varphi_{Z_m}\subset U$, and therefore the projection of $g$ to $G_{Z_m}$ is eventually trivial. This means exactly that $G_Z$ is the direct limit of the sequence $\left (G_{Z_n}\right )_n$.
\end{proof}
We now restrict to the case where $(X, \varphi)$ is a subshift over a finite alphabet $A$.
Recall that a subshift $X\subset A^\Z$ over a finite alphabet is called a \emph{subshift of finite type} if there exists $n\ge 1$ and a finite set $L\subset A^n$ of the set of words of length $n$ such that $X$ consists exactly of all sequences  whose subwords of length $n$ all belong to   $L$.

Assume that $X\subset A^\Z$ is an arbitrary subshift (not necessarily of finite type). For each $n\ge 1$ let $L_n\subset A^n$ be the collection of finite words of length $n$ that appear in some sequence of $X$, and let  $(X_n, \varphi_n)$ be the  subshift of finite type defined by $L_n$. This provides a nested sequence of subshifts:

\[X_1\supset X_2\supset \cdots \supset X=\bigcap_{n\ge 1} X_n.\]
In this specific setting, we rewrite Lemma \ref{l-limit} for further reference as follows:
\begin{lem}\label{p-direct-limit}
Let $(X, \varphi)$ be a  subshift, and  $\left ((X_n, \varphi_n)\right )_n$ be the sequence of subshifts of finite type defined above. Then the group $\PL(\varphi)$ is the direct limit of the sequence of groups $\left (\PL(\varphi_n)\right )_n$ with respect to the restriction homomorphisms $\PL(\varphi_1)\to \PL(\varphi_2)\to\cdots$. Similarly the group $\Tsf(\varphi)$ is the direct limit of the sequence of groups $\left (\Tsf(\varphi_n)\right )_n$.
\end{lem}

We point out the following consequence:
\begin{cor}
If $(X, \varphi)$ is a subshift which is not of finite type, the group $\Tsf(\varphi)$ is not finitely presentable. In particular for every infinite minimal subshift $(X, \varphi)$, the group $\Tsf(\varphi)$ is not finitely presentable. 
\end{cor}
\begin{proof}
If $(X, \varphi)$ is  not of finite type, the inclusion $X_n\supset X$ is strict for every $n\ge 1$, so that the natural epimorphism $\Tsf(\varphi_n)\to \Tsf(\varphi)$ has a non-trivial kernel. Since $\Tsf(\varphi)=\varinjlim \Tsf(\varphi_n)$ is the direct limit of finitely generated groups (by Theorem \ref{t-fg} applied to every $\Tsf(\varphi_n)$), the group $\Tsf(\varphi)$ cannot be finitely presented. 
Moreover a  subshift of finite type  has periodic orbits, thus it is never minimal (unless it is finite and consists of a single periodic orbit).
\end{proof}

We now turn to study Kazhdan's property $(T)$ for subgroups of $\PL(\varphi)$. For preliminaries on Kazhdan's property $(T)$, we refer to \cite{BHV}. We will freely use some basic facts, in particular that a discrete group with Kazhdan's property $(T)$ is finitely generated, that homomorphic images of a property $(T)$ group still have property $(T)$, and that  amenable groups with property $(T)$ are finite. 

In addition we will need the following result from \cite{Cor-partial} and \cite{LMT}. We denote by $\PL(\Sbb^1)$  the group of piecewise linear orientation-preserving homeomorphisms of the circle. 

\begin{thm} \label{t-PL-T-circle}
Every subgroup of $\PL(\Sbb^1)$ with Kazhdan's property $(T)$ is finite cyclic. 
\end{thm}

To proceed to the proof of Theorem \ref{t-T} we begin by considering  the case where $(X, \varphi)$ is a subshift of finite type. 
Recall that a subshift of finite type defined by the set $L\subset A^n$ is \emph{irreducible}  if it is not reduced to a single periodic orbit and if for every $w_1$ and $w_2\in L$, there exists a finite  word $u\in A^*$ which admits  $w_1$ as a prefix and $w_2$ as a suffix, and such that all subwords of $u$ of length $n$ belong to $L$.  It is  well-known and not difficult to see that an irreducible subshift of finite type  has a dense set of periodic orbits, and admits dense infinite orbits.

\begin{lem} \label{l-SFT-T}
Assume that $(X, \varphi)$ is an irreducible subshift of finite type. Then $\PL(\varphi)$ has no non-trivial subgroup with Kazhdan's property $(T)$.
\end{lem}
\begin{proof}
Let $H\le \PL(\varphi)$ be a subgroup with Kazhdan's property $(T)$. 
 The subshift $(X,\varphi)$  has a dense collection of periodic orbits, hence the suspension $(Y^\varphi,\Phi)$ has a dense collection of closed $\Phi$-orbits. The restriction of the action of $H$ to a closed orbit $\mathcal C$ gives rise to  a homomorphism $\rho_{\mathcal C}:H\to \PL(\T)$ to the group of PL homeomorphisms of the circle. Using Theorem \ref{t-PL-T-circle}, the image $\rho_{\mathcal C}(H)$ must be  finite cyclic. But the density of closed orbits implies that $H$ acts faithfully on the union of closed orbits.  Therefore $H$ embeds in an infinite product of finite cyclic groups. We deduce that $H$ is abelian (hence amenable), and since it has property $(T)$, it must be finite. However an irreducible subshift of finite type has a dense infinite orbit, so that by Proposition~\ref{p-lo} we deduce that $\PL(\varphi)$ is torsion-free and $H$ is actually trivial.
\end{proof}

\begin{lem} \label{l-minimal-T} Let $(X, \varphi)$ be an infinite minimal subshift. Then $\PL(\varphi)$ has no non-trivial subgroup with Kazhdan's property $(T)$. 
\end{lem}

\begin{proof}
Let $H\le \PL(\varphi)$ have Kazhdan's property $(T)$, in particular $H$ is finitely generated, let $S$ be a finite generating set of $H$. 
Let $\left ((X_n, \varphi_n)\right )_n$ be the associated sequence of  subshifts of finite type as explained above. Then each $(X_n, \varphi_n)$ is irreducible: indeed any two words $w_1$ and $w_2\in L_n$ appear infinitely often  in every sequence of $X$, as a consequence of the minimality of $(X, \varphi)$.  By Lemma \ref{p-direct-limit} we have $\PL(\varphi)=\varinjlim \PL(\varphi_n)$. Since the connecting homomorphisms are surjective, we can choose a finite set $S_1\subset \PL(\varphi_1)$ which projects bijectively to $S$. Let $H_1$ be the group generated by $S_1$ and $H_n \le \PL(\varphi_n)$ be its image. Then each $H_n$ is finitely generated and $H=\varinjlim H_n$.  By  \cite[Theorem 6.7]{Shalom} and \cite[\S 3.8.B]{RWRG}, a group with property $(T)$ cannot be a direct limit of finitely generated groups without property $(T)$. We deduce that $H_n$ must also have property $(T)$ for $n$ large enough. By Lemma \ref{l-SFT-T},  we deduce that $H_n$ is  trivial for $n$ large enough, and thus so is $H$.
\end{proof}

\begin{prop} \label{p-T-subshift}
For every subshift $(X, \varphi)$, every subgroup of   $\PL(\varphi)$ with Kazhdan's property~$(T)$ is finite  abelian.   
\end{prop}
\begin{proof}
Fix $H\le \PL(\varphi)$ with Kazhdan's property $(T)$. 
For $x\in X$ we let $\cO_x\subset Y^\varphi$ be the corresponding $\Phi$-orbit in the suspension, and let $\rho_x\colon H\to \Homeo_+(\cO_x)$ be the homomorphism obtained by restricting the $H$-action to $\cO_x$. Let us prove that, for every $x\in X$, the image of $\rho_x$ must  be  finite cyclic. As in Lemma \ref{l-SFT-T}, this implies that $H$ is abelian, hence finite by Kazhdan's property $(T)$. 
If $x$ is periodic, so that $\mathcal{O}_x$ is a closed orbit, then the image of $\rho_x$ is finite cyclic by Theorem \ref{t-PL-T-circle}. Thus we can assume that $x$ has an infinite $\varphi$-orbit (in this case, we will prove that the image of $\rho_x$ is trivial). 

Choose a closed minimal  $\varphi$-invariant subset $Z\subset X$  contained in the orbit-closure of $x$ (note that $Z$ could be reduced to a single finite orbit). Consider the restriction homomorphism $H\to \PL(\varphi|_Z)$ (Lemma \ref{l:surjection}). If $Z$ is a finite orbit, Theorem \ref{t-PL-T-circle} implies that its image is finite cyclic, and  if $Z$ is infinite  Lemma \ref{l-minimal-T} implies that its image is trivial.  In both cases, denote by $K\unlhd H$  its kernel, which has finite index in $H$. The group $K$  still has property $(T)$, and in particular it is finitely generated.  
 \begin{claim}
 There exists an open subset $V\supset Y^\varphi_Z$ such that $K$ fixes $V$ pointwise. 
 \end{claim}
 \begin{proof}[Proof of Claim]
 Assume first that $Z$ is infinite, so that $Y^\varphi_Z$ does not contain closed orbits.
 Fix $g\in K$. For every $y\in Y^\varphi_Z\subset Y^\varphi$, there exists a domain $U_{C \times I}\subset Y^\varphi$ containing $y$,  in restriction to which $g$  is given by a single PL homeomorphism $f:I\to f(I)$.  As $g$ is the identity in restriction to $Y^\varphi_Z \cap U_{C\times I}$ and $Y^\varphi_Z$ does not contain closed orbits, $f$ must necessarily be the identity on $I$, so that $g$ actually fixes $U_{C \times I}$ pointwise. Since $y$ was arbitrary, we have that $g$ is the identity in restriction to an open neighbourhood of $Y^\varphi_Z$. Finally, since $K$ is finitely generated, we can find an open neighbourhood of $Y^\varphi_Z$ which is fixed by all elements in some finite generating set of $K$ and the conclusion follows. 
 
 Assume now that $Z$ is a finite orbit of period $N>0$, so that $Y^\varphi_Z$ consists of a single closed orbit of length $N$. 
 Again fix $g\in K$. Let $y\in Y^\varphi_Z\subset Y^\varphi$ and choose a  domain $U_{C \times I}\subset Y^\varphi$ containing $y$, where $g$ is given by a single PL homeomorphism $f\colon I \to f(I)$. As $g$ belongs to the kernel of the restriction homomorphism (Lemma \ref{l:surjection}), the PL homeomorphism $f$ must coincide on $I$ with a translation by an integer $n_{g, y}$ which is a multiple of $N$. In other words, the germ of $g$ at the point $y$ coincides with the germ of the flow $\Phi^{n_{g, y}}$.  (Note that in particular the integer $n_{g, y}$ does not depend on the choice of the domain, as every neighbourhood of $y$  must intersect non-trivially the orbit $\mathcal{O}_x$ which is not periodic, so that the flow $\Phi$ acts freely on it.)  The map $g\mapsto n_{g, y}$ is  a homomorphism $K\to \Z$. Using that $K$ has Kazhdan's property $(T)$, this homomorphism must have trivial image, so that we must have $n_{g, y}=0$ for every $g\in K$. In particular, every $g$ must fix a neighbourhood of $y$, and we can repeat the previous reasoning.
 \end{proof}
 
 To conclude the proof of the proposition, let $V\supset Y^\varphi_Z$ be a neighbourhood fixed by $K$. Since $\cO_x$ accumulates on $Y^\varphi_Z$, the intersection $\cO_x\cap V$ is non-empty. It follows that the action of $\rho_x(K)$  on $\cO_x$ admits  global fixed points.  Assume towards a contradiction that $\rho_x(K)$ is non-trivial, and let $\xi\in \cO_x$ a point lying on the boundary of the set of global fixed points. Since the action of elements of $\rho_x(K)$ is piecewise linear, we have a non-trivial homomorphism $\rho_x(K)\to \R$  given by the logarithm of a one-sided derivative at $\xi$. This is in contradiction with the fact that $K$ has Kazhdan's property $(T)$.  We deduce that $\rho_x(K)$ is trivial. In particular, $\rho_x(H)$ is a homomorphic image of $H/K$, and therefore is finite cyclic. Since moreover $\rho_x(H)$ is torsion-free, we deduce that it is actually trivial.
This concludes the proof. 
\end{proof}

The assumption that $(X, \varphi)$ is a subshift is not essential, due to  Lemma \ref{l-fg-subshift}:

\begin{proof}[Proof of Theorem \ref{t-T}]
Let $(X, \varphi)$ be an arbitrary Stone system, and assume that $G\le \PL(\varphi)$ is a subgroup with Kazdan's property $(T)$. Then $G$ is finitely generated, so that by Lemma \ref{l-fg-subshift} we can find a subshift $(X', \varphi')$ and a homomorphism  $G\to \PL(\varphi')$ with abelian kernel. 
 By Proposition~\ref{p-T-subshift}, the image of this homomorphism is finite. Therefore $G$ is virtually abelian, and thus property $(T)$ implies that it is  finite. 
\end{proof}

\section{Isomorphisms correspond to flow equivalences} \label{s-isomorphisms}

We now prove Theorem \ref{t-isom}, which will be obtained  by combining two results. The first is the following:
\begin{thm}[(Ben Ami and Rubin \cite{BAR})] \label{t-BAR}
For $i=1, 2$, let $G_i$ be a fragmentable group of homeomorphisms of a locally compact Hausdorff space $Y_i$ whose action on $Y_i$ has no global fixed points. Then for every group isomorphism $\rho\colon G_1\to G_2$ there exists a homeomorphism $q\colon Y_1\to Y_2$ such that $\rho(g)=q\circ g\circ q^{-1}$.
\end{thm}

The second is a  characterisation of flow equivalence due to Parry and Sullivan \cite{PS}, which we recall now. Given a Stone system $(X, \varphi)$, a \emph{clopen transversal} is a clopen subset $C\subset X$ which intersects all $\varphi$-orbits. It follows that for every point $x\in X$ the first entry time \[T_C(x)=\min\{n\ge 1\colon \varphi^n(x)\in C\}\]
is finite and locally constant as a function of $x$, and that the corresponding \emph{first return map} $\varphi_C\colon C\to C, x\mapsto \varphi^{T_C(x)}(x)$ is a homeomorphism. 
\begin{thm}[(Parry and Sullivan \cite{PS})] \label{t-PS}
Two Stone systems $(X_i, \varphi_i)$, for $i=1, 2$, are flow equivalent if and only if they admit clopen transversals $C_i\subset X_i$ whose first return maps are topologically conjugate.
\end{thm}

See  also \cite{BCE} for a more recent and detailed exposition of this result.

\begin{proof}[Proof of Theorem \ref{t-isom}]
For $i=1, 2,$ let $(X_i, \varphi_i)$ be a Stone system with suspension flow $(Y_i,\Phi_i)$, and assume that each $(X_i, \varphi_i)$ admits a dense orbit.  Assume that the groups $\Tsf(\varphi_i)$ are isomorphic. By Proposition \ref{p-fragmentable}, the groups $\Tsf(\varphi_i)$ are fragmentable. Thus by Theorem \ref{t-BAR} there exists a homeomorphism $q\colon Y_1\to Y_2$ which conjugates the actions. Note that $q$ must send $\Phi_1$-orbits to $\Phi_2$-orbits, as these are the connected components of $Y_1$ and $Y_2$, respectively.  Moreover it must simultaneously preserve or reverse the orientations of all orbits, as this is uniquely determined by whether it does so on any dense orbit of $Y_1$. Thus $(X_1, \varphi_1)$ is flow equivalent to  $(X_2, \varphi_2)$ or to $(X_2, \varphi_2^{-1})$. 

Conversely assume that $(X_1, \varphi_1)$ is flow equivalent to $(X_2, \varphi_2)$ and let us show that $\Tsf(\varphi_1)$ and $\Tsf(\varphi_2)$ are isomorphic (note that it is enough to consider this case,  since  $\Tsf(\varphi_2)=\Tsf(\varphi_2^{-1})$). To this end we need to check that a homeomorphism $q\colon Y_1\to Y_2$ realising the flow equivalence can be chosen so that it conjugates $\Tsf(\varphi_1)$ to $\Tsf(\varphi_2)$. By Theorem \ref{t-PS}, the two systems admit  clopen transversals whose first return maps are topologically conjugate, and we can identify these two clopen transversals with a same subset $C\subset X_1$ and $C\subset X_2$ with common first return map equal to $\psi\colon C\to C$; let $T^{(i)}\colon C\to \N$ be the first return time with respect to $\varphi_i$. Choose a partition $C=\bigsqcup_{j=1}^r C_j$ into clopen subsets, such that in restriction to every $C_j$, the functions $T^{(1)}$ and $T^{(2)}$ are both constant and equal to integers $n_j$ and $m_j$ respectively. We can write $Y_1$ and $Y_2$ as the union of the closures of domains $Y_1=\bigcup_{j=1}^r\overline{U^{(1)}_{C_j\times (0, n_j)}}$ and $Y_2=\bigcup_{j=1}^r\overline{U^{(2)}_{C_j\times (0, m_j)}}$. 
For every $j$, the choice of a homeomorphism $f_j\colon [0, n_j]\to [0,m_j]$ induces a homeomorphism $q_j\colon \overline{U^{(1)}_{C_j\times (0, n_j)}}\to \overline{U^{(2)}_{C_j\times (0, m_j)}}$, and by construction the homeomorphisms $q_j$ are compatible on the boundary and give rise to a homeomorphism $q\colon Y_1\to Y_2$. If we choose the homeomorphisms $f_j$ to be PL dyadic, the resulting homeomorphism $q$ will conjugate $\Tsf(\varphi_1)$ to $\Tsf(\varphi_2)$.
\end{proof}

{\appendix

\section{Topology on $\Hsf_0(\varphi)$}\label{s-polish}
In this appendix we endow the group $\Hsf_0(\varphi)$  with a  topology defined by a complete metric. When $X$ is metrisable, this topology is separable, and thus makes $\Hsf_0(\varphi)$ a Polish group. We then observe that the subgroup $\Tsf(\varphi)$ is dense in it with respect to this topology.  This is analogous to the fact that Thompson's group $T$ is dense in $\Homeo_+(\Sbb^1)$ for the $C^0$ topology \cite[Corollary A5.8]{Bi-St}. As application, we prove  Proposition \ref{p-inf-gen}.

 For any closed interval $I\subset\R$, the group $\Homeo_+(I)$ is a Polish group with respect to the distance $\delta(g, h)=\sup_{x\in I} |f(x)-g(x)|+\sup_{x\in I}|f^{-1}(x)-g^{-1}(x)|$. The topology on $\Hsf_0(\varphi)$ is a natural generalization of this. 

By \cite[Theorem 3.1]{BCE}, for every $g\in \Hsf_0(\varphi)$ there exists a continuous function $\tau\colon Y^\varphi\to \R$ such that for every $y\in Y^\varphi$ we have $g(y)=\Phi^{\tau(y)}(y)$. We will say that $\tau$ is a \emph{displacement function} for $g$. 
 Note that the value of $\tau(y)$ is uniquely determined by $g$ and $y$, provided  $y$ belongs to a non-closed orbit. Thus, by continuity, a displacement function for $g$ is uniquely determined by $g$  provided the union of non-periodic orbits of $(X, \varphi)$ is dense, that is, if $(X, \varphi)$ is topologically free. However, if this assumption is not satisfied, an element $g$ can admit more than one displacement function. If $\tau$ and $\tau'$ are two different displacement functions for $g$, then $\sigma:=\tau-\tau'$ takes integer values, and its value at every point $y\in Y^\varphi$ is a multiple of the period of $y$. Conversely given any continuous function $\sigma\colon Y^\varphi\to \Z$ with this property and a displacement function $\tau$ for $g$, then $\tau':=\tau+\sigma$ is still a displacement function for $g$.
  
 Finally note that if $\tau_1$ is a displacement function for $g_1$ and $\tau_2$ for $g_2$, then a displacement function for $g_3=g_2g_1$ is given by the formula
\begin{equation}\label{e-cocycle}
\tau_3(y)=\tau_2(g_1(y))+\tau_1(y).
\end{equation}
For any two $g$ and  $h\in \Hsf_0(\varphi)$, set
\[d_+(g, h)=\inf_{\tau, \lambda} \sup_{y\in Y^\varphi} |\tau(y)-\lambda(y)|\]
where the infimum is taken over all possible displacement functions $\tau$ and $\lambda$, for $g$ and $h$ respectively. This defines a distance on $\Hsf_0(\varphi)$, which is however not complete.  To ensure completeness,  set $d_-(g, h):=d_+(g^{-1}, h^{-1})$, and 
\[d(g, h)=d_{+}(g, h)+d_-(g, h)\] 
which is again a distance on $\Hsf_0(\varphi)$.
\begin{prop}
The distance $d$ is complete and defines a group topology on $\Hsf_0(\varphi)$. 

\end{prop}
\begin{proof} 
Let $(g_n)_n$ be a Cauchy sequence. It follows that we can find a sequence of displacement functions $(\tau_n)_n$ for $(g_n)_n$ and $(\lambda_n)_n$ for $(g_n^{-1})_n$ that are Cauchy with respect to the distance of uniform convergence. Therefore, $(\tau_n)_n$ and $(\lambda_n)_n$ converge uniformly to continuous functions $\tau, \lambda \colon Y^\varphi\to \R$ respectively, and we can define continuous maps $h, k\colon Y^\varphi\to Y^\varphi$ by $h(y)=\Phi^{\tau(y)}(y)$ and $k(y)=\Phi^{\lambda(y)}(y)$ respectively. Passing to the limit in the equality
\[y=g_ng_{n}^{-1}(y)=\Phi^{\tau_{n}(\Phi^{\lambda_n(y)}(y))+\lambda_n(y)}(y)\]
 we see that $hk(y)=y$, and similarly $kh(y)=y$. Thus $h$ and $k$ are homeomorphisms inverse to each other, and $g_n$ tends to $h\in \Hsf_0(\varphi)$ as $n\to \infty$. 

Let us now check that $d$ defines a group topology. It is clear that the map $g\mapsto g^{-1}$ is continuous. Let $(g_n)_n$ and $(h_n)_n$ be sequences converging to $g$ and $h$ respectively, and let us check that $(g_nh_n)_n$ tends to $gh$ as $n\to \infty$. Let us show that $d_+(g_nh_n, gh)\to 0$. Fix displacement functions $\tau$ for $g$, and $\lambda$ for $h$, as well as two sequences $(\tau_n)_n$ and $(\lambda_n)_n$ of displacement functions for $(g_n)_n$ and $(h_n)_n$ that converge to $\tau$ and $\lambda$ uniformly. Set $\delta(y)=\tau(h(y))+\lambda(y)$ and $\delta_n(y)=\tau_n(h_n(y))+\lambda_n(y)$, these define displacement functions for $gh$ and for $g_nh_n$ respectively (see \eqref{e-cocycle}). We have
\[|\delta_n(y)-\delta(y)|\le |\tau_n(h_n(y))-\tau(h_n(y))|+|\tau(h_n(y))-\tau(h(y))|+|\lambda_{n}(y)-\lambda(y)|.\]
The first and third summand tend to zero uniformly on $y$ because  $\tau_n\to \tau$ and $\lambda_n\to \lambda$ uniformly. Moreover the second term also tends to zero uniformly by uniform continuity of $\tau$, and because  $h_n(y)=\Phi^{{\lambda_n(y)}}(y)$ tends to $h(y)$ uniformly. Thus $\sup_y |\delta_n(y)-\delta(y)|$ tends to zero, so that $d_+(g_nh_n, gh)\to 0$.  Similarly one proves that $d_-(g_nh_n, gh)\to 0$, showing that $d(g_nh_n, gh)\to 0$.
\end{proof}

\begin{prop}
The group $\Tsf(\varphi)$ is a dense subgroup of $\Hsf_0(\varphi)$. In particular if $X$ is metrisable, the group $\Hsf_0(\varphi)$ is Polish. 
\end{prop}
\begin{proof}
The proof proceeds by  reducing to the classical remark that Thompson's group $F$ is dense in $\Homeo_+([0,1])$ \cite[Corollary A5.8]{Bi-St}.

Let us  show that  every $g\in \Hsf_0(\varphi)$ is in the closure of $\Tsf(\varphi)$. Assume first that $g$ is supported in a domain $U_{C\times I}$. Fix $\varepsilon>0$.  For every $x\in C$ the element $g$ fixes the fiber $\pi_{C,I}(\{x\}\times I\})$ and acts on it by a homeomorphism  $f_x\in \Homeo_+(I)$ which varies continuously with respect to $x$.  Thus we can find a clopen partition $C=\bigsqcup_{i=1}^r C_i$ such that whenever $x$ and $y\in C_i$, we have that $\delta(f_{x}, f_{y})<\varepsilon/2$, where $\delta$ is the distance in $\Homeo_+(I)$. For $i=1,\ldots, r$, choose $x_i\in C_i$ and an element $h_i\in F_I$ such that $\delta(h_i, f_{x_i})<\varepsilon/2$. If we identify $h_i$ with the corresponding element of $F_{C_i , I}$ and set $h=h_1\cdots h_r\in \Tsf(\varphi)$, we have $d(g, h)<\varepsilon$. Since $\varepsilon$ is arbitrary, we deduce that $g$ is in the closure of $\Tsf(\varphi)$. 

Next, assume that $d(g, \id)<1/8$. We consider the domains $U_{X\times (-1/4, 1/4)}\subset U_{X\times (-1/2, 1/2)}$. Then $g(U_{X\times (-1/4, 1/4)})\subset U_{X\times (-1/2, 1/2)}$. Find $g_1\in \Hsf_0(\varphi)$ supported in $U_{ X\times (-1/2, 1/2)}$ which coincides with $g$ in restriction to $U_{X\times (-1/4, 1/4)}$. Then $g_2:=g g_1^{-1}$ is supported in $U_{X\times (1/4, 3/4)}$ and therefore $g=g_2g_1$ is in the closure of $\Tsf(\varphi)$ by the previous case. 

Finally since $\Hsf_0(\varphi)$ is connected, it is generated by any neighbourhood of the identity. Therefore every element $g\in \Hsf_0(\varphi)$ is in the closure of $\Tsf(\varphi)$, that is, $\Tsf(\varphi)$ is dense. 

If $X$ is metrisable, it has only countably many clopen subsets.  Moreover there are  countably many dyadic intervals in $\R$, and countably many dyadic PL homeomorphisms between them. It follows from Definition \ref{d-group} (using compactness) that in this case the group $\Tsf(\varphi)$ is countable, and therefore $\Hsf_0(\varphi)$ is separable.
\end{proof}

We now prove Proposition~\ref{p-inf-gen} for an arbitrary Stone system $(X, \varphi)$ (see also Remark \ref{r-p-inf-gen-minimal} for the case where $(X, \varphi)$ is minimal). 
\begin{proof}[Proof of Proposition~\ref{p-inf-gen}] 
Let $K\le \Tsf(\varphi)$ be the subgroup generated by the groups $F'_{C, J}$. Fix $g\in \Tsf(\varphi)$  and let us show that $g\in K$. 
 Assume first that $d(g, \id)<1/4$. Let $X_0\subset Y^\varphi$ be the projection of $X\times\{0\}$ and cover it by disjoint domains  $U_{C_1\times I_1}, \ldots, U_{C_r\times I_r}$ with $0\in I_i$ so that in restriction to each of them, $g$ is given by a dyadic PL  homeomorphism $f_i\colon I_i\to f_i(I_i)$.  Since $-1/4<f_i(0)<1/4$, we can choose the intervals $I_i$ small enough so that  each closure $\overline{I_i\cup f_i(I_i)}$ is contained in $(-1/4, 1/4)$.  Choose  $k_i\in F_{(-1/4, 1/4)}$ that extends $f_i\vert_{I_i}$. Let $h_i\in F_{C_i, J_i}$ correspond to $k_i$, and $h=h_1\cdots h_r$.  Replace $g$ by $h^{-1}g$. After doing this, we obtain that $g$  fixes pointwise a domain of the form $U_{X \times I}$ for some $I\ni 0$. Equivalently, $g$ is supported in a domain of the form $U_{X\times J}$ for some $J\subset (0, 1)$. Now the conclusion follows from Lemma \ref{l-compactness-prod}.
 
 Assume now that $g$ is arbitrary. Since $\Hsf_0(\varphi)$ is connected and hence generated by any neighbourhood of the identity,  we can write $g=g_1\cdots g_r$ for some $g_i\in \Hsf_0(\varphi)$ with $d(g_i, \id)<1/4$. By density of $\Tsf(\varphi)$ we can choose $h_i\in \Tsf(\varphi)$  close enough to $g_i$ so that $d(h_i, \id)<1/4$ and the element $h=h_1\cdots h_r$ satisfies $d(gh^{-1}, \id)<1/4$. It follows from the previous case that $h$ and $gh^{-1}$ both belong to $K$ and thus $g\in K$.
 \end{proof}
}
{\acknowledgements

We thank Bertrand Deroin, who extensively advertised  the idea   to study  groups of self flow equivalences of  almost-periodic flows in the context of left-orderable groups. 
NMB thanks Kostya Medynets and Volodia Nekrashevych for discussions on presentations on topological full groups. The idea of proof of Theorem \ref{t-fp} is  inspired by  these discussions.  We are also grateful to Slava Grigorchuk, Fran\c cois Le Maître, Victor Kleptsyn and Samuel Petite for the many interesting comments on the first version of this paper. We also thank the referee for the careful reading and relevant suggestions.

This work has been written during a visit of MT to ETHZ. He thanks the institute for the welcoming atmosphere. He also thanks INRIA Lille for the semester of d\'el\'egation during the academic year 2018/19.}

\end{document}